\RequirePackage{etex}
\documentclass[ECP,preprint]{ejpecp}

\usepackage[authoryear]{natbib}
\usepackage{amsmath,amssymb,graphicx,url}
\theoremstyle{remark}
\newtheorem{rem}{Remark}%

\usepackage[framemethod=default]{mdframed}
\usepackage{etoc} %
\usepackage{fancyhdr}
\usepackage{nicefrac}
\usepackage{amsmath}

\newmdenv[topline=false,rightline=false,bottomline=false,nobreak=false]{proofaside}

\usepackage{pifont}
\hypersetup{
    urlcolor=blue
}
\usepackage{array}
\usepackage{mathtools}
\usepackage{amsmath,amssymb}
\usepackage{url}
\usepackage{comment}
\usepackage{graphicx}
\usepackage{enumitem}
\usepackage{wrapfig}
\usepackage{framed}
\usepackage{caption}
\usepackage{xspace}
\usepackage{graphicx}
\usepackage{booktabs}
\usepackage{enumitem}

\usepackage[capitalize]{cleveref}  

\let\oldproposition\proposition
\renewcommand{\proposition}{%
  \crefalias{theorem}{nprop}%
  \oldproposition}
\let\oldrem\rem
\renewcommand{\rem}{%
  \crefalias{theorem}{nrem}%
  \oldrem}
\let\oldcorollary\corollary
\renewcommand{\corollary}{%
  \crefalias{theorem}{ncor}%
  \oldcorollary}
\let\oldlemma\lemma
\renewcommand{\lemma}{%
  \crefalias{theorem}{nlem}%
  \oldlemma}
\crefname{nrem}{Rem.}{Rems.}
\crefname{rem}{Rem.}{Rems.}
\crefname{theorem}{Thm.}{Thms.}
\crefname{nlem}{Lem.}{Lems.}
\crefname{ncor}{Cor.}{Cors.}
\crefname{proposition}{Prop.}{Props.}
\crefname{nprop}{Prop.}{Props.}
\crefname{assumption}{Assump.}{Assumps.}
\crefname{talign}{}{}
\crefname{section}{Sec.}{Secs.}
\crefname{appendix}{App.}{Apps.}
\crefname{equation}{}{}

\usepackage{mathrsfs}
\usepackage{autonum} 
\usepackage[colorinlistoftodos]{todonotes}

\newcommand{\sig}{\sigma}

\def\balign#1\ealign{\begin{align}#1\end{align}}
\def\baligns#1\ealigns{\begin{align*}#1\end{align*}}
\def\balignat#1\ealign{\begin{alignat}#1\end{alignat}}
\def\balignats#1\ealigns{\begin{alignat*}#1\end{alignat*}}
\def\bitemize#1\eitemize{\begin{itemize}#1\end{itemize}}
\def\benumerate#1\eenumerate{\begin{enumerate}#1\end{enumerate}}

\newenvironment{talign*}
 {\csname align*\endcsname}
 {\endalign}
\newenvironment{talign}
 {\csname align\endcsname}
 {\endalign}

\def\balignst#1\ealignst{\begin{talign*}#1\end{talign*}}
\def\balignt#1\ealignt{\begin{talign}#1\end{talign}}
\newcommand{\qtext}[1]{\quad\text{#1}\quad} 
 
\newcommand{\stext}[1]{\ \text{#1}\ } 
\newcommand{\sstext}[1]{\ \ \text{#1}\ \ }

\let\originalleft\left
\let\originalright\right
\renewcommand{\left}{\mathopen{}\mathclose\bgroup\originalleft}
\renewcommand{\right}{\aftergroup\egroup\originalright}

\def\Renyi{R\'enyi\xspace}
\def\tinycitep*#1{{\tiny\citep*{#1}}}
\def\tinycitealt*#1{{\tiny\citealt*{#1}}}
\def\tinycite*#1{{\tiny\cite*{#1}}}
\def\smallcitep*#1{{\scriptsize\citep*{#1}}}
\def\smallcitealt*#1{{\scriptsize\citealt*{#1}}}
\def\smallcite*#1{{\scriptsize\cite*{#1}}}

\def\mbb#1{\mathbb{#1}}
\def\mc#1{\mathcal{#1}}
\def\mrm#1{\mathrm{#1}}

\def\reals{\mathbb{R}} %

\def\naturals{\mathbb{N}} %

\def\<{\left\langle} %
\def\>{\right\rangle}

\def\defeq{\triangleq} %
\def\half{\frac{1}{2}}

\def\norm#1{\|{#1}\|} %
\newcommand{\twonorm}[1]{\norm{#1}_2} %
\def\E{\mbb{E}} %

\def\P{\mbb{P}} %

\def\Var{\mrm{Var}} %

\def\Cov{\mrm{Cov}} %

\newcommand{\Unif}{\textnormal{Unif}}

\newcommand{\eqdist}{\stackrel{d}{=}}

\def\indep{\perp\!\!\!\perp} %
\newcommand{\iid}{\textrm{i.i.d.}\@\xspace}

\ifdefined\nonewproofenvironments\else
\ifdefined\ispres\else
\fi

\newcommand{\ncref}[1]{\cref{#1}: \nameref*{#1}} %

\newcommand{\pcref}[1]{\texorpdfstring{Proof of \ncref{#1}}{}} %

\newcommand{\Rsig}[1][\sigma]{\hyperref[Rsig]{\color{black}{R_{#1}}}} %

\newcommand{\wassbd}[1][\sigma]{\hyperref[cabris]{\color{black}{\omega^R_p(#1)}}} %
\newcommand{\KRsig}[1][\sigma]{\hyperref[KRsig]{\color{black}{K_{R,\sig}}}} 
\newcommand{\Rsigsqd}[1][\sigma]{\hyperref[Rsig]{\color{black}{R_{#1}^2}}} 
\newcommand{\Rsigpowfour}[1][\sigma]{\hyperref[Rsig]{\color{black}{R_{#1}^4}}}

\newcommand{\score}[1][S]{\mathfrak{s}_{#1}}

\newcommand{\density}[1][S]{f_{#1}}

\newcommand{\Cinfc}{C^{\infty}_c}

\newcommand{\Lp}[1][p]{L^{#1}}

\newcommand{\X}{\mathbf{X}}
\graphicspath{{src/figs/},{figs/}}

\newcommand{\mytitle}{Bounding Hellinger Distance with Stein's Method}
\newcommand{\myshorttitle}{\mytitle}
\newcommand{\myabstract}{%
This work introduces a new, explicit bound on the Hellinger distance between a continuous random variable and a Gaussian with matching mean and variance. 
As example applications, we derive a quantitative Hellinger central limit theorem and efficient concentration inequalities for U-statistics.
}
\newcommand{\mykeywords}{%
Hellinger distance; Stein's method; central limit theorem; efficient concentration inequalities; U-statistics} %

\SHORTTITLE{\myshorttitle}

\TITLE{\mytitle}

\AUTHORS{%
  Morgane~Austern\footnote{Harvard University, United States of America.
    \EMAIL{morgane.austern@gmail.com}\thanks{MA was funded by ONR  N000142112664}}%
  \and %
  Lester~Mackey\footnote{Microsoft Research New England,
    United States of America. \EMAIL{lmackey@microsoft.com}}} %

\KEYWORDS{\mykeywords} %

\AMSSUBJ{60F05; 60F10} %

\SUBMITTED{November 5, 2024} %
\ACCEPTED{} %

\VOLUME{0}
\YEAR{2023}
\PAPERNUM{0}
\DOI{10.1214/YY-TN}

\ABSTRACT{\myabstract}

\begin{document}
\section{Introduction}
We introduce a new, explicit bound (\cref{general_bound}) on the Hellinger distance \citep{hellinger1909neue} 
between a continuous random variable and a Gaussian with matching mean and variance.
\cref{general_bound}, based on Stein's method of exchangeable pairs \citep{stein1986approximate}, enables quantitative Hellinger central limit theorems (CLTs, \cref{ustat}) and efficient concentration inequalities (\cref{dependent_concentration}) for U-statistics. 
Notably, \cref{ustat} also avoids the unknown constants, spectral gap conditions \citep{artstein2004rate},
Stein kernel assumptions \citep{ledoux2015stein}, and exponential dependence on Kullback-Leibler divergence \citep{bobkov2014berry,bobkov2013rate} present in existing entropic CLTs for sums of independent random variables.

\section{Bounding Hellinger Distance to a Gaussian}\label{sec:general_bound}
We begin by presenting a general upper bound on the Hellinger distance 
between a continuous random variable and a Gaussian with matching mean and variance. 
Throughout, we let $Z$ denote a standard normal random variable with Lebesgue density $\varphi$
and $\Cinfc$ denote the set of infinitely differentiable compact functions from $\reals$ to $\reals$.
Further, for a random variable $S$ and $p\geq 1$, 
we write $C^p$ for the set of $p$-times continuously differentiable functions from $\reals$ to $\reals$, define $\norm{S}_p^p\defeq \E[|S|^p]$, 
and
let $\Lp(S) \defeq \{ \textup{measurable\ } f : \reals\to\reals \mid \E[|f(S)|^p]<\infty\}$.
Finally, we say that $S$ admits a \emph{(first-order) score function} $\score[S]$ or a \emph{second-order score function} %
$I_S$ if $\score[S], I_S \in \Lp[1](S)$  satisfy 
\begin{talign}\label{eq:score-assumps}
\E[\score[S](S)\Psi(S)]=-\E[\Psi'(S)]
\qtext{and}
\E[I_S(S)\Psi(S)]=\E[\Psi''(S)]
\qtext{for all} \Psi\in  \Cinfc.
\end{talign}
Notably, by \cref{It}, $\score[S] = f'/f$ and $I_S=f''/f$ whenever those quantities exist in $L^1(S)$. 
\begin{theorem}[Bounding Hellinger distance to a Gaussian]\label{general_bound}

Let $\mathbf{X}$ be random variable taking values in a Polish-space $\mathcal{X}$ and $g:\mathcal{X}\rightarrow\mathbb{R}$ be a measurable function for which $S\defeq g(\mathbf{X})$ has convex support $\Omega_S\subseteq\mathbb{R}$, $\E[S]=0$, 
$\Var(S)=1$, Lebesgue density $f_S\in C^1$, %
and {first and second-order score functions} %
$\score[S],I_S \in\Lp[2](S)$ satisfying $\twonorm{\score[S](S) S}<\infty$.
Fix any $D\in(0,1)$ and any anti-symmetric function $F(\cdot,\cdot):\mathcal{X}^2\rightarrow\mathbb{R}$ and, for each $t \geq -\half\log(1-D)$, 
define $S^t\defeq g(\X^t)$ where $\X^t$ is an exchangeable copy of $\mathbf{X}$ and
\begin{talign}\label{eq:exchangeable-pair-assumption}
\sum_{k=1}^{\infty}\frac{e^{-kt}k}{\sqrt{k-1}!\sqrt{1-e^{-2t}}^k}\|\E[(S^t-S)^{k-1}F(\X^t,\mathbf{X}) \mid S]\|_2< \infty.
\end{talign}
Then, for any $s\in\reals$,
we may bound the Hellinger distance
\begin{talign}\label{oui}
 H(S,Z) 
 &\defeq (\half\int_{\mathbb{R}}(\sqrt{f_S(x)}-\sqrt{\varphi(x)})^2dx)^{1/2}\\
 &\le\frac{1}{\sqrt{2}}\int_0^{\infty}\|I_{Z_t}(Z_t)-I_{Z}(Z_t)+Z_t(\score[Z_t](Z_t)-\score[Z](Z_t))\|_2dt\label{eq:prelim_hellinger_bound}
   \\&\le \frac{1}{\sqrt{2}}
   \{-\frac{1}{2}\log(1-D)(1+\|\score[S](S)S\|_2)
    +\|\score[S](S)\|_2(\sqrt{\frac{D}{1-D}}-\rm{tan}^{-1}(\sqrt{\frac{D}{1-D}}))\\
&+\frac{D}{1-D}\|I_S(S)\|_2
    + 
\int_{-\frac{1}{2}\log(1-D)}^{\infty}\|\frac{e^{-2t}}{{1-e^{-2t}}}h_2(Z)-\frac{e^{-t}}{\sqrt{1-e^{-2t}}}Sh_1(Z)-\tilde\tau_{t,s}\|_2dt\}\label{eq:second_hellinger_bound}
\end{talign}
in terms of the Gaussian score functions $\score[Z](z) \defeq -z$ and $I_Z(z) \defeq z^2-1$, the interpolants
\begin{talign}
Z_t\defeq e^{-t}S+\sqrt{1-e^{-2t}}Z 
    \sstext{with} 
    Z\indep S, 
    \sstext{density}
f_{Z_t},
    \sstext{and}
    (\score[Z_t],I_{Z_t}) \defeq (\frac{\partial_x{f_{Z_t}}}{f_{Z_t}}, 
\frac{\partial_x^2f_{Z_t}}{f_{Z_t}})
\end{talign}
for $t > 0$,  
and the Hermite sum 
\begin{talign}%
\tilde\tau_{t,s}
    &\defeq
\frac{se^{-t}}{\sqrt{1-e^{-2t}}}\E[F(\X^t,\mathbf{X})\mid S]h_1(Z)+\sum_{k=2}^{\infty}\frac{s e^{-kt}}{k!\sqrt{1-e^{-2t}}^k}\E[(S^t-S)^{k-1}F(\X^t,\mathbf{X})\mid S]h_k(Z)
\end{talign}
where $h_k(x)\defeq (-1)^k e^{x^2/2} \frac{\partial^k}{\partial x}e^{-x^2/2}$ designates  the $k$-th Hermite polynomial for $k\in\naturals$.
\end{theorem}
\begin{remark}[Parameter choices]
In applications, we will choose $D=\frac{1}{\sqrt{n}}$
and $s$, $\X^t$, and $F$ to satisfy %
$s\E[F(\X^t,\mathbf{X})\mid S]\approx -S$ 
and 
$\frac{s}{2}\E[(S^t-S)F(\X^t,\mathbf{X})\mid S]\approx 1$,
so that 
$\|\frac{e^{-2t}}{{1-e^{-2t}}}h_2(Z)-\frac{e^{-t}}{\sqrt{1-e^{-2t}}}Sh_1(Z)-\tilde\tau_{t,s}\|_2$ is small and 
the overall bound \cref{eq:second_hellinger_bound} is $O(\frac{\log(n)}{\sqrt{n}})$.
\end{remark}
\begin{remark}[Integrability]
The integrability condition \cref{eq:exchangeable-pair-assumption} ensures that $\twonorm{\tilde\tau_{t,s}}<\infty$ and holds, for example, whenever $S^t-S$ and $F(\X^t, \X)$ are bounded, as in \cref{ustat}.
\end{remark}

\begin{remark}[Second-order score bound]
The second-order score can often be bounded in terms of first-order score functions, as, by \cref{It2}, if $S\overset{d}{=}S_1+S_2$ for independent $S_1, S_2$ with score functions $\score[S_1]\in\Lp[2](S_1)$ and $\score[S_2]\in \Lp[2](S_2)$ then 
$S$ admits a second-order score function $I_S$ satisfying 
\begin{talign}\label{eq:second-order-score-bound}
\|I_S(S)\|_2\le \|\score[S_1](S_1)\|_2\|\score[S_2](S_2)\|_2.  
\end{talign}
\end{remark}
Our proof of \cref{general_bound} in  \cref{proof_general_bound}  relies on 
Stein's method of exchangeable pairs \citep{stein1986approximate}
and 
interpolation along the Ornstein-Uhlenbeck semigroup \citep[Chap.~4]{pavliotis2014stochastic}. %
The first component takes inspiration from the arguments used by \citet{bonis2020stein} to bound Wasserstein distances. 
To obtain an effective bound, \cref{general_bound} requires four inputs: the exchangeable copy $\X'$, the anti-symmetric function $F$, and the scalars $D$ and $s$. We next demonstrate how to select these inputs  when the target $S$ is a U-statistic.

\subsection{Application: A Quantitative Hellinger CLT for U-statistics}\label{ustat}
Given an \iid sequence of random variables $\X\defeq (X_i)_{i=1}^n$ taking values in $\mathcal{X}$ 
and a symmetric, bounded, measurable  function $u:\mathcal{X}^2\rightarrow\mathbb{R}$, consider the canonical U-statistic
\begin{talign}\label{eq:ustat}
U_n = g(\X) \defeq  \frac{1}{\sqrt{n}(n-1)}\sum_{1\le i\ne j\le n}u(X_i,X_j)-\mathbb{E}[u(X_i,X_j)].
\end{talign} 
Whenever $U_n$ admits first and second-order score functions we can bound its Hellinger distance to a Gaussian using  \cref{general_bound} 
with $D=\frac{1}{\sqrt{n}}$, $s=n$, and, for each $t>0$,
\begin{talign}
&\!\!\!\!\!F(\X^t,\X)
    \defeq
\frac{1}{2}(g(\X^t)-g(\X))+g_1(\X^t)-g_1(\X) 
    \ \text{for}\ 
\X^t
    \defeq
(X_1,\dots,X_{I-1},X'_I,X_{I+1},\dots,X_n),   
 \\
&(X_i')_{i=1}^n      
    \eqdist 
\X,\ 
I\sim \Unif([n]), 
    \stext{and}
g_1(\X) 
    \defeq 
\frac{1}{\sqrt{n}}\sum_{i\le n}\E[u(X_i, X_i')| X_i]-\E[u(X_i, X_i')],
\label{eq:pair_choices}
\end{talign}
and $(\X,(X_i')_{i=1}^n,I)$  mutually independent. 
Here we adopt the shorthand $[n]\defeq \{1,\dots,n\}$.
\begin{theorem}[Quantitative Hellinger CLT for U-statistics]\label{ustattheorem} 
Suppose that the U-statistic $U_n$ \cref{eq:ustat} admits first and second-order score functions $\score[U_n], I_{U_n}\in L^2(U_n)$ with $|u(X_1,X_2)| \stackrel{a.s.}{\leq} R$ and  
$4\Cov(u(X_1,X_2),u(X_1,X_3))=1$. 
Then, for all $n\ge 2$, 
\begin{talign}  \label{exact_comp}&
H(U_n,Z)
    \le 
\frac{3\alpha  R^3}{\sqrt{2n}}(R+\frac{R\log({4\sqrt{n}})}{2\sqrt{n}}
    +
\frac{4}{n^{1/4}})e^{\frac{2R^2}{\sqrt{n}}}
    +
\frac{\log(n)}{\sqrt{2}\sqrt{n-1}} \|u(X_1,X_2)\|_4^2
  (16+2\sqrt{2}+\frac{8}{\sqrt{n-1}})\\
  &-\frac{1}{2\sqrt{2}}\log(1-\frac{1}{\sqrt{n}})(1+\|\score[U_n](U_n)U_n\|_2)
  +
\frac{1}{\sqrt{2}(\sqrt{n}-1)}\|I_{U_n}(U_n)\|_2
   +
\frac{1}{3\sqrt{2}(\sqrt{n}-1)^{3/2}}\|\score[U_n](U_n)\|_2
\end{talign}
for $\alpha \defeq e^{19/300}\pi^{1/4}$. 
Hence, if $\max(\|\score[U_n](U_n)\|_2,\|I_{U_n}(U_n)\|_2,\|U_n\score[U_n](U_n)\|_2)=O(\log(n))$, 
\begin{talign}
H(U_n,Z)\lessapprox\frac{\log(n)}{\sqrt{n}}
\|u(X_1,X_2)\|_4^2 
(8\sqrt{2}+2)
\end{talign}
where the notation $a_n\lessapprox b_n$  signifies that $a_n\le b_n+o(b_n).$
\end{theorem}
\cref{ustattheorem}, proved in \cref{proof_ustattheorem}, yields a U-statistic Hellinger CLT rate of order $\frac{\log(n)}{\sqrt{n}}$ whenever the U-statistic score functions have suitably bounded moments. 
This matches, up to a logarithmic factor, the entropic CLT rates obtained for sums of independent random variables \citep{artstein2004rate,ledoux2015stein,bobkov2014berry,bobkov2013rate}. Moreover, U-statistic score functions can often be controlled by their input variable score functions as the next example, proved in \cref{proof_example}, illustrates. %
\begin{corollary}[Quantitative Hellinger CLT for chaos statistics]\label{example}
Suppose that the \iid sequence $(X_i)_{i\ge 1}$ has $\score[X_1]\in L^4(X_1)$ and $X_1 \in [a,b]$ for $a > 0$. 
Then the U-statistic $U_n$ \cref{eq:ustat} with $u(x,y) = xy$ has $\max( \|\score[U_n](U_n)\|_2,\|I_{U_n}(U_n)\|_2,\|U_n\score[U_n](U_n)\|_2)=O(1)$
and 
    \begin{talign}
    H(\frac{1}{\sigma}U_n,Z)\lessapprox\frac{\log(n)}{\sqrt{n}}\frac{\E[X_1^4]}{\sigma^2} (8\sqrt{2}+2)
    \qtext{for}
    \sigma^2\defeq 4\E[X_1]^2\Var(X_1).
    \end{talign} 
\end{corollary}
\subsection{Application: Efficient Concentration for U-statistics}\label{sec:concentration}

As a second application of \cref{general_bound}, we develop efficient concentration inequalities for U-statistics. 
Let $\Phi$ represent the cumulative distribution function of a standard Gaussian and $\Phi^c = 1-\Phi$. 
For an asymptotically normal sequence $(S_n)_{n \geq 1}$ with $\E(S_n) = 0$ and $\Var(S_n) = 1$, 
\citet{austern2022efficient} defined an efficient concentration inequality as an upper bound on $\P(S_n > u)$ that converges to the asymptotically exact value $\Phi^c(u)$ 
as $n \to \infty$.  
The authors also derive such inequalities for scaled sums of independent variables.
Our Hellinger bounds enable efficient concentration for more general random variables via the following  relation, proved in \cref{proof_concentration}.

\begin{lemma}[Hellinger concentration]\label{concentration} 
If $S$ has a Lebesgue density, then, for all $u\in\reals$, %
\begin{talign}
|\sqrt{\P(S > u)} - \sqrt{\Phi^c(u)}|
    \vee
|\sqrt{\P(|S| > u)} - \sqrt{2\Phi^c(u)}|
&\leq
\sqrt{2}H(S,Z).
\end{talign}
\end{lemma}

For example, \cref{example,concentration} together immediately imply the following efficient concentration inequality for the chaos U-statistics $U_n$ of  \cref{example}.

\begin{theorem}[Efficient U-statistic concentration]\label{dependent_concentration}
Under the assumptions of \cref{example}, 
\begin{talign}
|\sqrt{\P(U_n > \sigma u)} - \sqrt{\Phi^c(u)}|
\lessapprox \frac{\log(n)}{\sqrt{n}}\frac{\sqrt{2}\E(X_1^4)}{\sigma^2}(8\sqrt{2}+2) 
\qtext{for all} u\in\reals.
\end{talign} 
\end{theorem}

\subsection{Related Work} 
For sums of \iid observations, \citet{barron1986entropy} established the first entropic CLT using the de Bruijn identity.
Rates were then established in various \Renyi divergences by  \citet{artstein2004rate,ledoux2015stein,bobkov2014berry,bobkov2013rate,goldfeld2020limit,yao2025symmetric,bobkov2025central}. See \cite{bobkov2025renyi} for a recent survey of those results. To the best of our knowledge, no rates for the Hellinger CLT have previously been established for U-statistics or other dependent sums. %
Stein's method has been used to bound the Kolmogorov and Wasserstein distances to Gaussianity with $O(\frac{1}{\sqrt{n}})$ rates for U-statistics \citep{ross2011fundamentals}. %
However, neither distance controls the Hellinger distance without additional assumptions.  When the characteristic function $\phi_S$ of $S$ satisfies $\int t|\phi_S(t)|dt<\infty$,  \citet{pratelli2018convergence,janson2024quantitative} bound the total variation (TV) distance $d_{TV}$ by the Kolmogorov distance. Together these results yield suboptimal $O(\frac{1}{n^{1/4}})$ Hellinger rates as $H^2(S,Z)\le d_{TV}(S,Z)$. %
We control the Hellinger distance to Gaussianity using score function conditions and establish the first $O(\frac{\log n}{\sqrt{n}})$ rates for non-degenerate U-statistics.

\section{\pcref{general_bound}}
\label{proof_general_bound}

We begin by establishing the  inequality \cref{eq:prelim_hellinger_bound}.
Fix any $T<\infty$. 
By definition we have
\begin{talign}
H(S,Z_T)^2
    &=
1-\int_{\Omega_S}\sqrt{f_S(x)}\sqrt{f_{Z_T}(x)}dx.
\end{talign}
Fix any 
$x\in \Omega_S$.
Since 
the function $t\rightarrow \sqrt{f_{Z_t}(x)}$ is 
continuous on $[0,\infty)$ and continuously  differentiable on $(0,\infty)$ 
by \cref{espoir},
the fundamental theorem of calculus \citep[Thm.~1.32]{knapp2007basic}
implies that $\int_0^T\partial_t\sqrt{f_{Z_t}(x)}dt=\sqrt{f_{Z_T}(x)}-\sqrt{f_S(x)}$. 
Moreover, since $\int_{0}^{T}\int_{\Omega_S}\sqrt{f_S(x)}|\partial_t\sqrt{f_{Z_t}(x)}|dx\,dt < \infty$ by \cref{marreee}, Fubini's theorem \citep[Cor.~13.9]{Schilling_2005} implies that
\begin{talign}
    H(S,Z_T)^2&=-\int_{0}^{T}\int_{\Omega_S}\sqrt{f_S(x)}\partial_t\sqrt{f_{Z_t}(x)}dx\,dt.
\end{talign}
    
We will re-express this in terms of information-theoretic quantities. Write $h_t\defeq{f_{Z_t}/}{\varphi}$. By the Fokker-Planck equation \citep[Thm.~4.1]{pavliotis2014stochastic}, 
we have \begin{talign}\label{oma}\partial_t\sqrt{f_{Z_t}(x)}&=\frac{1}{2\sqrt{f_{Z_t}(x)}}\partial_t{f_{Z_t}(x)}
=\frac{\varphi(x)}{2\sqrt{f_{Z_t}(x)}}\partial_t\frac{f_{Z_t}(x)}{\varphi(x)}
=\frac{\varphi(x)}{2\sqrt{f_{Z_t}(x)}}\partial_th_t(x)\\
&=\frac{\varphi(x)}{2\sqrt{f_{Z_t}(x)}}[\partial_x^2{h_t}(x)-x\partial_x{h_t}(x)]
=\frac{f_{Z_t}(x)}{2\sqrt{f_{Z_t}(x)}}[\frac{\partial_x^2 h_t(x)}{h_t(x)}-x\frac{\partial_x h_t(x)}{h_t(x)}].
\end{talign}
Defining $\Delta_t \defeq \|I_{Z_t}(Z_t)-h_2(Z_t)+Z_t(\score[Z_t](Z_t)-h_1(Z_t))\|_2$, we therefore obtain that 
\begin{talign}
   H(S,Z_T)^2 &=-\int_{0}^{T}\int_{\Omega_S}\sqrt{f_S(x)}\frac{f_{Z_t}(x)}{2\sqrt{f_{Z_t}(x)}}[\frac{\partial_x^2 h_t(x)}{h_t(x)}-x\frac{\partial_x h_t(x)}{h_t(x)}]dt\\
   &=%
   -\int_0^{T}\E[\frac{\sqrt{f_S(Z_t)}}{2\sqrt{f_{Z_t}(Z_t)}}[\frac{\partial_x^2 h_t(Z_t)}{h_t(Z_t)}-Z_t\frac{\partial_x h_t(Z_t)}{h_t(Z_t)}]]dt
     \\&\overset {(a)}{=}-\frac{1}{2}\int_0^{T}\E[[\frac{\sqrt{f_S(Z_t)}}{\sqrt{f_{Z_t}(Z_t)}}-1][\frac{\partial_x^2 h_t(Z_t)}{h_t(Z_t)}-Z_t\frac{\partial_x h_t(Z_t)}{h_t(Z_t)}]]dt
      \\&\overset{(b)}{\le} \frac{1}{2} \sup_{t\le T}\|1-\frac{\sqrt{f_S(Z_t)}}{\sqrt{f_{Z_t}(Z_t)}}\|_2\int_0^{T}\Delta_t dt
     =  \sup_{t\leq T}\frac{1}{\sqrt{2}}H(S,Z_t)\int_0^{T}\Delta_t dt, \label{eq:HSST-bound}
\end{talign}
where (b) uses Cauchy-Schwarz and the identities
\begin{talign}\label{eq:derivloghtS}
\frac{\partial_x h_t(x)}{h_t(x)}&=\frac{\partial_xf_{Z_t}(x)}{f_{Z_t}(x)}+x 
=\score[Z_t](x)+x
=\score[Z_t](x)-\score[Z](x) 
\qtext{and} \\ 
\frac{\partial_x^2 h_t(x)}{h_t(x)}&=\frac{{\partial_x^2f^{S}_{t}}(x)}{f_{Z_t}(x)}-x^2+1+2x\frac{\partial_x h_t(x)}{h_t(x)}
\\&=I_{Z_t}(x)-I_Z(x)+2x(\score[Z_t](x)-\score[Z](x))
=I_{Z_t}(x)+x^2+1+2x\score[Z_t](x),\label{eq:secondderivloghtS}
\end{talign} 
and (a) follows from the identities 
\begin{talign}
\E[\frac{\partial_x^2 h_t(Z_t)}{h_t(Z_t)}-Z_t\frac{\partial_x h_t(Z_t)}{h_t(Z_t)}]\overset{(a_1)}=\E[I_{Z_t}(Z_t)+1+Z_t \score[Z_t](Z_t)]\overset{(a_2)}{=}0.
\end{talign}
The equality $(a_1)$ is a consequence of the identities \cref{eq:derivloghtS,eq:secondderivloghtS}, and $(a_2)$ follows as, by our score assumption \cref{eq:score-assumps}, we have $\E[I_{Z_t}(Z_t)]=0$ and $\E[Z_t \score[Z_t](Z_t)]+1=0$.

Finally, since, for each $x\in\Omega_S$, $f_{Z_t}(x) = \frac{\E[\varphi((x - Se^{-t})/\sqrt{1-e^{-2t})}]}{\sqrt{1-e^{-2t}}} \to \varphi(x)$ as $t\to\infty$ and $|f_{Z_t}(x)|\leq 1/\sqrt{1-e^{-2}}$ for all $t\geq 1$, the dominated convergence theorem \citep[Thm.~11.2]{Schilling_2005} and \cref{eq:HSST-bound} yield
\begin{talign}
&H(S,Z)^2
    =
H(S,Z_T)^2+\int_{\Omega_S}\sqrt{f_S(x)}(\sqrt{f_{Z_T}(x)}-\sqrt{\varphi(x)})dx \\
    &\leq
\limsup_{T\to\infty} (H(S,Z_T)^2+\int_{\Omega_S}\sqrt{f_S(x)}(\sqrt{f_{Z_T}(x)}-\sqrt{\varphi(x)})dx) \\
    &\leq
\sup_{T\geq 0} H(S,Z_T)^2+\int_{\Omega_S}\sqrt{f_S(x)}(\limsup_{T\to\infty}\sqrt{f_{Z_T}(x)}-\sqrt{\varphi(x)})dx 
    \leq
\half (\int_0^{\infty}\Delta_t dt)^2.
\end{talign}

To obtain the final bound \cref{eq:second_hellinger_bound}, 
we fix any $D\in(0,1)$ and $s>0$, %
decompose the integral
\begin{talign}
&\int_0^{\infty}\Delta_t dt
= \int_0^{-\frac{1}{2}\log(1-D)}\Delta_t dt
+\int_{-\frac{1}{2}\log(1-D)}^{\infty}\Delta_t dt
\defeq (b_1)+(b_2), 
\end{talign}
and bound each component separately. 

By \cref{It}, $\score[Z_t](Z_t)$ is the almost-surely unique %
 $\Lp[1]$ 
 random variable satisfying 
\begin{talign}\label{mdt2}
    \E[\score[Z_t](Z_t)\Psi(Z_t)]=-\E[\Psi'(Z_t)]
    \qtext{for all}
    \Psi\in \Cinfc.
\end{talign} 
Moreover, by the tower property and the independence of $Z$ and $S$, we have
\begin{talign}
\E[\E[\score(S)\mid Z_t]\Psi(Z_t)]
    =
\E[\score(S)\Psi(e^{-t}S+\sqrt{1-e^{-2t}}Z)]
    =
-e^{-t}\E[\Psi'(Z_t)],
    \ \ \forall
    \Psi\in \Cinfc.
\end{talign}
Therefore we must have
\begin{talign}\label{jst1}
\score[Z_t](Z_t)
    =
e^t\E[\score(S)\mid Z_t]
\qtext{almost surely.}
\end{talign}
Similarly, by \cref{It}, 
$I_{Z_t}(Z_t)$ is the almost-surely unique 
$\Lp[1]$ %
random variable 
satisfying 
\begin{talign}
    \E[I_{Z_t}(Z_t)\Psi(Z_t)]=\E[\Psi''(Z_t)]
    \stext{for all}
    \Psi\in \Cinfc.
\end{talign} 
Moreover, the tower property and independence of $Z$ and $S$, imply that, 
for all $\Psi\in C_c^{\infty}$, 
\begin{talign}
\E[\Psi(Z_t)e^{2t}\E[I_S(S)\mid Z_t]]&= e^{2t}\E[\Psi(Z_t)I_S(S)] 
    =\E[\Psi''(e^{-t}S+\sqrt{1-e^{-2t}}Z){}]
    =\E[\Psi''(Z_t){}].
\end{talign}
Hence we obtain that, almost surely, 
\begin{talign}\label{mdt}
I_{Z_t}(Z_t)&={e^{2t}}\E[I_S(S)\mid Z_t]
\stext{and} 
    I_{Z_t}(Z_t)+Z_t\score[Z_t](Z_t)=e^{t}\E[e^{t}I_S(S)+\score[S](S)Z_t\mid Z_t].
\end{talign}
Therefore, the tower property, Jensen's inequality, and the triangle inequality yield 
\begin{talign}
\|I_{Z_t}(Z_t)+1+Z_t\score[Z_t](Z_t)\|_2&= \|\E[e^{2t}I_S(S)+e^t\score[S](S)Z_t+1\mid Z_t]\|_2\\&\le \|e^{2t}I_S(S)+e^t\score[S](S)Z_t+1\|_2\\&\le 1+e^{2t}\|I_S(S)\|_2+\|\score[S](S)S\|_2+e^t\sqrt{1-e^{-2t}}\|\score[S](S)\|_2%
\end{talign} 
and hence 
\begin{talign}
(b_1)
     &\le-\frac{1}{2}\log(1-D)(1+\|\score[S](S)S\|_2)+\frac{D}{1-D}\|I_S(S)\|_2+\|\score[S](S)\|_2(\sqrt{\frac{D}{1-D}}-\rm{tan}^{-1}(\sqrt{\frac{D}{1-D}})).
\end{talign}

It remains to bound $(b_2)$. By \cref{It} and the tower property, $\score[Z](z)=-z$ and $I_Z(z)=z^2-1$ are first and second-order score functions for $Z$ satisfying, for all $\Psi\in  \Cinfc$, 
\begin{talign}
   \E[\Psi(Z_t)\frac{1}{\sqrt{1-e^{-2t}}}\E[Z\mid Z_t]] &=   \E[\Psi(Z_t)\frac{Z}{\sqrt{1-e^{-2t}}}]
   =\E[\Psi'(Z_t)] \qtext{and} \\
   \E[\Psi(Z_t)\frac{1}{{1-e^{-2t}}}\E[Z^2-1\mid Z_t]]&=    \E[\Psi(Z_t)\frac{Z^2-1}{{1-e^{-2t}}}]
   =\E[\Psi''(Z_t)].
\end{talign} 
 Using \cref{It} once more we therefore obtain the almost sure identities 
\begin{talign}
I_{Z_t}(Z_t)
    =
\E[\frac{1}{1-e^{-2t}}(Z^2-1)\mid Z_t]
    \qtext{and}
-\score[Z_t](Z_t)-Z_t
    =
\E[\frac{e^{-2t}}{\sqrt{1-e^{-2t}}}Z-e^{-t}S\mid Z_t].
\end{talign} 

To conclude we will make use of the following lemma proved in \cref{proof_tau_0}. 
\begin{lemma}[Hermite sum conditional means]\label{tau_0} 
 Under the assumptions of \cref{general_bound},  
 \begin{talign}
\E[\tilde\tau_{t,s}\mid Z_t]=0 
 \qtext{almost surely for all} 
 t \geq -\half\log(1-D).
\end{talign}
\end{lemma}
By \cref{tau_0} we have 
$\E[\tilde\tau^{S}_t\mid Z_t]=0$ almost surely. 
Thus, with probability $1$,
\begin{talign}
I_{Z_t}(Z_t)
    &-I_{Z}(Z_t)
=\E[\frac{Z^2}{1-e^{-2t}}-(Z_t)^2\mid Z_t]-\frac{e^{-2t}}{1-e^{-2t}}
  \\&=
\E[\frac{Z}{\sqrt{1-e^{-2t}}}(\frac{Z}{\sqrt{1-e^{-2t}}}-Z_t)+Z_t(\frac{Z}{\sqrt{1-e^{-2t}}}-Z_t)\mid Z_t]-\frac{e^{-2t}}{1-e^{-2t}}
        \\
    &=
\E[\frac{Z}{\sqrt{1-e^{-2t}}}(\frac{e^{-2t}}{\sqrt{1-e^{-2t}}}Z-e^{-t}S)\mid Z_t]-\frac{e^{-2t}}{1-e^{-2t}}+Z_t\E [(\frac{e^{-2t}}{\sqrt{1-e^{-2t}}}Z-e^{-t}S)\mid Z_t] \\
&= \E[\frac{e^{-2t}}{{1-e^{-2t}}}(Z^2-1)-\frac{e^{-t}}{\sqrt{1-e^{-2t}}}SZ - \tilde\tau_{t,s}\mid Z_t]+Z_t(-\score[Z_t](Z_t)-Z_t),
\end{talign}
and Jensen's inequality therefore implies that
\begin{talign}
(b_2)
&= \int_{-\frac{1}{2}\log(1-D)}^{\infty}\|\E[\frac{e^{-2t}}{{1-e^{-2t}}}(Z^2-1)-\frac{e^{-t}}{\sqrt{1-e^{-2t}}}SZ\mid Z_t]\|_2dt
\\&\le \int_{-\frac{1}{2}\log(1-D)}^{\infty}\|\frac{e^{-2t}}{{1-e^{-2t}}}(Z^2-1)-\frac{e^{-t}}{\sqrt{1-e^{-2t}}}SZ-\tilde \tau_{t,s}\|_2dt 
\end{talign}
as
$\|\tilde \tau_{t,s}\|_2\le
\sum_{k=1}^{\infty}\frac{se^{-kt}k}{\sqrt{k-1}!\sqrt{1-e^{-2t}}^k}\|\E[(S^t-S)^{k-1}F(\X^t,\mathbf{X})\mid S]\|_2< \infty$
by the independence of $Z$ and $S$ and 
the Hermite polynomial inequality  $\twonorm{h_k(Z)}\leq \sqrt{k!}$ \citep[Lem.~3]{bonis2020stein}.
\subsection{\pcref{tau_0}}\label{proof_tau_0}
Fix any $\Psi\in \Cinfc$, define $(P_t\Psi)(x)\defeq \E[\Psi(e^{-t}x+\sqrt{1-e^{-2t}}Z)]$, and, for each $\psi \in \Cinfc$, let $\psi^{(k)}$ denote its $k$-th derivative. 
Since $S\indep Z$, 
the Hermite polynomials satisfy  $(P_t\Psi)^{(k)}(S)=\frac{\E[h_k(Z)\Psi(Z_t)\mid S]}{e^{kt}\sqrt{1-e^{-2t}}^k}$ almost surely for each $k$ by \citet[(16)]{bonis2020stein}. Hence, we may invoke Fubini's theorem \citep[Cor.~13.9]{Schilling_2005} (twice), the tower property (twice), the analyticity of $(P_t\Psi)'$  (by \citet[(24)]{bonis2020stein} and \citet[Prop. 1.2.12]{krantz2002primer}), the exchangeability of $(\X,\X^t)$, and the anti-symmetry of $F$ to find that
\begin{talign}
    &\E \left[ \sum_{k \ge 2}  \frac{ \E[F(\X^t,\mathbf{X})(S^t-S)^{k-1}\mid S] }{(k-1)! e^{kt} \sqrt{1 - e^{-2t}}^{k}} h_k(Z)\Psi(Z_t)   \right] 
    \overset{(a)}{=}
    \sum_{k \ge 2}\E \left[  \frac{\E[F(\X^t,\mathbf{X}) (S^t-S)^{k-1} \mid S] \E[h_k(Z)\Psi(Z_t)\mid S]}{(k-1)! e^{kt}\sqrt{1 - e^{-2t}}^{k}}    \right]
    \\
    &= \sum_{k \ge 2} \E [  F(\X^t,\mathbf{X})  \frac{ (S^t-S)^{k-1}}{(k-1)! } (P_t \Psi)^{(k)}(S) ]  
    \overset{(b)}{=} \E [  F(\X^t,\mathbf{X})  \sum_{k \ge 2} \frac{ (S^t-S)^{k-1}}{(k-1)! } (P_t \Psi)^{(k)}(S) ] \\
    &=
       \E  \left[F(\X^t,\mathbf{X}) ( (P_t \Psi)'(S^t) - (P_t \Psi)'(S))   \right] 
    =  -2\E  \left[ F(\X^t,\mathbf{X}) (P_t \Psi)'(S^t)   \right].
\end{talign}
Using Fubini's theorem to swap expectation and sum in (a) and (b) is justified as
\begin{talign}\label{eq:integrable_tautilde}
\sum_{k=1}^{\infty}\frac{e^{-kt}}{(k-1)!\sqrt{1-e^{-2t}}^k}\E[\E[|S^t-S|^{k-1}|F(\X^t,\mathbf{X})|\mid S]\, |h_k(Z)|\, |\Psi(Z_t)|] < \infty
\end{talign}
by the boundedness of $\Psi$, the independence of $Z$ and $S$,  
the Hermite polynomial inequality  $\E[|h_k(Z)|]\leq \sqrt{k!}$ \citep[Lem.~3]{bonis2020stein}, and  \cref{eq:exchangeable-pair-assumption}. 
Parallel reasoning implies that 
\begin{talign} &%
\E \left[ \E[F(\X^t,\mathbf{X})\mid S] h_1(Z)\Psi(Z_t)\right]
    = %
    e^t \sqrt{1-e^{-2t}}\E\left[   F(\X^t,\mathbf{X}) (P_t \Psi)'(S)\right].
\end{talign}
Hence, for all $\Psi\in  \Cinfc$,  $\E[\tilde\tau_{t,s} \Psi(Z_t)]=0$.
Since $\E[\tilde \tau_{t,s}\mid Z_t] \in L^1$ by \cref{eq:integrable_tautilde} with $\Psi\equiv 1$, \cref{It} implies that $\E[\tilde \tau_{t,s}\mid Z_t] =0$ almost surely.%
\section{\pcref{ustattheorem}}\label{proof_ustattheorem}
\cref{ustattheorem} follows from the following more detailed statement with $D = 1$.
A yet tighter bound could be obtained by optimizing over $D$, which trades off the sizes of $\beta$ and $\beta_0$.

\begin{theorem}[Detailed Hellinger CLT for U-statistics]\label{exact_V}
Under the assumptions of \cref{ustattheorem},  fix  
any $D\in(0,\sqrt{n})$ and define
$\kappa\defeq \frac{R^2}{D\sqrt{n}}$ and 
$B_{n} 
    \defeq  (\frac{8}{n-1}+\frac{16+2\sqrt{2}}{\sqrt{n-1}})\|u(X_1,X_2)\|_4^2$.
Then
\begin{talign}  \label{exact_comp}
\sqrt{2}&  H(U_n,Z)
    \le %
\beta + 
\beta_{0}
\qtext{for}
    \\
\beta
    &\defeq
\frac{3R^4\alpha  e^{2\kappa}}{\sqrt{n}}(\frac{\sqrt{1-\frac{D}{\sqrt{n}}}}{D}
    +
\frac{\log(\frac{4\sqrt{n}}{D})}{2\sqrt{n}})
    +
\frac{6\alpha R^3 e^{2\kappa}}{(n\kappa)^{3/4}\sqrt{D}\sqrt{n-1}}
    +
\frac{2\sqrt{3}\alpha R^3e^{2\kappa}}{n^{3/4}\sqrt{D}}
    +
\frac{{B_{n}}\log(\frac{n}{D^2})}{2\sqrt{2}}
    \qtext{and}\\
\beta_{0} 
    &\defeq
-\frac{1}{2}\log(1-\frac{D}{\sqrt{n}})(1+\|\score[U_n](U_n)U_n\|_2)
   +\frac{D}{\sqrt{n}-D}\|I_{U_n}(U_n)\|_2+\frac{1}{3}(\frac{D}{\sqrt{n}-D})^{3/2}\|\score[U_n](U_n)\|_2. %
\end{talign}
\end{theorem}
\begin{proof}
Our result follows directly from three lemmas. The first
uses \cref{general_bound} to bound $H(U_n, Z)$ in terms of the exchangeable pair selections of \cref{eq:pair_choices}. %

{\begin{lemma}[Exchangeable pair bound on $H(U_n,Z)$]\label{lemma_2}
Instantiate the assumptions of \cref{exact_V}, 
and, for each $t>0$, define the Hermite sum $\tilde\tau_{t,s}$ as in \cref{general_bound} with 
the $g$, $F$, $\X$, and $\X^t$ selections of 
\cref{eq:pair_choices} and $s=n$. Then
\begin{talign}\label{exchangeable_pair_hellinger_bound}
\sqrt{2}  H( U_n,Z)
  &\le 
\int_{-\frac{1}{2}\log(1-\frac{D}{\sqrt{n}})}^{\infty}\|\frac{e^{-2t}}{{1-e^{-2t}}}(Z^2-1)-\frac{e^{-t}}{\sqrt{1-e^{-2t}}}U_nZ-\tilde\tau_{t,s}\|_2\,dt
+ \beta_0.
\end{talign}
\end{lemma}}
\begin{proof}
We simply invoke \cref{general_bound} and the estimate $x-\tan^{-1}(x)\leq {x^3}{/3}$ for $x\geq 0$.
\end{proof}

The second, proved in \cref{proof_pair_moment_bounds}, bounds several relevant exchangeable pair quantities. %
\begin{lemma}[Exchangeable pair moment bounds]\label{pair_moment_bounds}%

Under the remaining assumptions of \cref{lemma_2},
for any $t\ge-\frac{1}{2}\log(1-\frac{R^2}{\kappa n})$  with $c(t)\defeq\sqrt{\kappa(1-e^{-2t})}$, 
assumption \cref{eq:exchangeable-pair-assumption} holds,
\begin{talign}\label{second}
&\|n\E[F(\X^t,\X)\mid U_n]+U_n\|_{2}=0,\quad
   \|\frac{n}{2}\E[F(\X^t,\X)(g(\X^t)-g(\X))\mid U_n]-1\|_{2}
    \le 
B_{n}, \\
\label{higher}
&n\|\E[F(\X^t,\X)(g(\X^t)-g(\X))^{k-1}\mid U_n]\|_{2}
     \le  
\frac{2^k3c(t)^{k-4}R^4}{2n}, \stext{$\forall$ even $k\ge 4$, and}\\&n\|\E[F(\X^t,\X)(g(\X^t)-g(\X))^{k-1}\mid U_n]\|_{2}
     \le  
\frac{2^k3c(t)^{k-3}R^3}{2n}(1+\frac{\sqrt{k}}{\sqrt{n-1}}), \stext{$\forall$ odd $k\ge 3$.}
\end{talign}
\end{lemma}

The third, proved in \cref{proof_porcel}, bounds the large-time component of \cref{exchangeable_pair_hellinger_bound} using \cref{pair_moment_bounds}.
\begin{lemma}[Large-time Hellinger bound]
    \label{porcel}
Under the assumptions of \cref{lemma_2}, 
\begin{talign}&
\int_{-\frac{1}{2}\log(1-\frac{D}{\sqrt{n}})}^{\infty}\|\frac{e^{-2t}}{{1-e^{-2t}}}(Z^2-1)-\frac{e^{-t}}{\sqrt{1-e^{-2t}}} U_nZ-\tilde\tau_{t,s}\|_{2}dt
    \le
\beta.
    \end{talign}
\end{lemma}
\end{proof}

\subsection{\pcref{pair_moment_bounds}}\label{proof_pair_moment_bounds}
First, the independence of $(\X,(X_i')_{i=1}^n,I)$ and the tower property imply that
\begin{talign}
&\E\left[ g_1(\X^t)-g_1(\X) \mid \X \right] 
    = 
\frac{1}{\sqrt{n}}\E[u_1(X_I')  - u_1(X_I) \mid \X]
= - \frac{1}{n} g_1(\X)
\qtext{for}
u_1(x) \defeq \E[u(X_1,x)], \\
&\E[g(\X^t)-g(\X) \mid \X]
    =
\frac{2}{\sqrt{n}(n-1)}\E[\sum_{j\ne I }u(X_I',X_j) - u(X_I,X_j)\mid \X]
    = 
\frac{2}{n} g_1(\X) - \frac{2}{n} U_n, 
    \qtext{and}\\
&\E[F(\X^t,\X) \mid U_n]=\E[\E[F(\X^t,\X) \mid \X]\mid U_n] = -\frac{1}{n}U_n.
\end{talign}

This yields the first claim of \cref{second}.
To prove the second, we write
\begin{talign}&
    \E[(g(\X^t)-g(\X))F(\X^t,\X)\mid \X]
    \\&=   \frac{1}{2} \E[(g(\X^t)-g(\X))(g(\X^t)-g(\X))\mid \X]+    \E[(g(\X^t)-g(\X))(g_1(\X^t)-g_1(\X))\mid \X]\\&=\frac{2}{n^2(n-1)^2}\sum_{i\le n}\E[(\sum_{j\ne i}u(X'_i,X_j)-u(X_i,X_j))^2\mid \X]
    \\&+\frac{2}{n^2(n-1)}\sum_{i\le n}\sum_{j\ne i}\E[(u(X'_i,X_j)-u(X_i,X_j))(u_1(X'_i)-u_1(X_i))\mid \X]
    =A+B
\end{talign} 
and bound each term  successively. 
Defining $u_2 \defeq \E[u(X_1,X_2)]$, we find that
\begin{talign}
B
&=\frac{2}{n^2(n-1)}\sum_{i\le n}\sum_{j\ne i}\E[u(X'_i,X_j)u_1(X'_i)\mid X_i]-\frac{2}{n^2(n-1)}\sum_{i\le n}\sum_{j\ne i}u_1(X_j)u_1(X_i)
\\&\qquad-\frac{2}{n^2(n-1)}\sum_{i\le n}\sum_{j\ne i}u_2u(X_i,X_j)
+\frac{2}{n^2(n-1)}\sum_{1<i\le n}\sum_{j\ne i}u_1(X_i)u(X_i,X_j).
\end{talign}
Therefore, $n\mathbb{E}[B]=4(\E[u_1(X_1)^2]-u_2^2)$.
Moreover, the triangle inequality implies that
\begin{talign}
  &\!\!\!\!n  \|B-\E[B]\|_2
  \le B_1+B_2+B_3+B_4\defeq \frac{2}{n}\|\sum_{i\le n}\E[u(X'_1,X_i)u_1(X'_1)\mid X_i]-\E[u_1(X_1)^2]\|_2
  \\&+\frac{2}{n(n-1)}\|\sum_{\substack{i\le n\\j\ne i}}u_1(X_i)u(X_i,X_j)-\E[u_1(X_1)^2]\|_2
  \\&+\frac{2}{n(n-1)}\|\sum_{\substack{i\le n\\j\ne i}}u_1(X_i)u_1(X_j)-u_2^2\|_2+\frac{2}{n(n-1)}\|\sum_{\substack{i\le n\\j\ne i}}u_2u(X_i,X_j)-u_2^2\|_2.
\end{talign}

As the random variables $(X_i)_{i=1}^n$ are \iid, we have, by Jensen's inequality,
\begin{talign}
B_1
    &=
\frac{2}{\sqrt{n}}\sqrt{\Var(\E[u(X'_1,X_2)u_1(X_1')|X_2])}
    \le 
\frac{2}{\sqrt{n}}\|u(X_1,X_2)\|_4^2.
\end{talign}
Similarly, the triangle inequality, the \iid assumption, and Jensen's inequality imply
\begin{talign}
    B_2&\le
    \frac{2}{n(n-1)}\|\sum_{\substack{i\le n\\j\ne i}}u_1(X_i)u(X_i,X_j)-u_1(X_i)^2\|_2+  \frac{2}{n}\|\sum_{\substack{i\le n}}u_1(X_i)^2-\E(u_1(X_1)^2)\|_2    %
   \\&= \frac{2}{(n-1)}\|u_1(X_1)\sqrt{\Var(\sum_{{j>1}}u(X_1,X_j)|X_1)}\|_2+\frac{2}{\sqrt{n}}\sqrt{\Var(u_1(X_1)^2)}
    \\&\overset{(b)}{\le}\frac{2}{\sqrt{n-1}}\|u_1(X_1)\sqrt{\E[u(X_1,X_2)^2|X_1)}\|_2+\frac{2}{\sqrt{n}}\|u(X_1,X_2)\|_4^2
   \le \frac{4}{\sqrt{n-1}}\|u(X_1,X_2)\|_4^2.
\end{talign}
We use the same tools to bound the third term
\begin{talign}
  B_3
  &\le \frac{2}{n(n-1)}\|\sum_{\substack{i\le n}}u_1(X_i)\sum_{\substack{j\neq i}}(u_1(X_j)-u_2)\|_2+\frac{2(n-1)u_2}{n(n-1)}\|\sum_{\substack{i\le n}}(u_1(X_i)-u_2)\|_2 %
    \\
    &\le \frac{2}{n-1}\|u_1(X_1)\sum_{\substack{1<i\le n}}(u_1(X_i)-u_2)\|_2+\frac{2u_2}{\sqrt{n}}\sqrt{\Var(u_1(X_1))}
    \le\|u(X_1,X_2)\|^2_4\frac{4}{\sqrt{n-1}}.
\end{talign}
Finally, since $\Cov(u(X_i,X_j),u(X_{i_2},X_{j_2}))=0$ when all indices $(i,j,i_2,j_2)$ are distinct, 
\begin{talign}
  B_4&= \frac{2u_2}{n(n-1)}\sqrt{\Var(\sum_{i\le n,j\ne i}u(X_i,X_j))}
  {\le
  }\frac{4u_2\sqrt{n-2}}{\sqrt{n(n-1)}}\|u(X_1,X_2)\|_2
   \le \frac{4}{\sqrt{n}}\|u(X_1,X_2)\|_4^2.
\end{talign}
Hence, 
$n  \|B-\E[B]\|_2
\le \frac{14}{\sqrt{n-1}}\|u(X_1,X_2)\|_4^2$. 
Next, we note that 
\begin{talign}
&A
    =
A_1+A_2
    \defeq
\frac{2}{n^2(n-1)^2}\sum_{i\le n}\sum_{j\ne i}\E[u(X'_i,X_j)^2\mid \X]
    -2u_1(X_j)u(X_i,X_j)
    +u(X_i,X_j)^2 \\
&+\frac{2}{n^2(n-1)^2}\sum_{i\le n}\sum_{j_1\ne j_2\ne i}\E[(u(X'_i,X_{j_1})-u(X_i,X_{j_1}))(u(X'_i,X_{j_2})-u(X_i,X_{j_2}))\mid \X]
\end{talign}
so that 
$n \E[A]=\frac{4}{n-1}(\mathbb{E}[u(X_1,X_2)^2]-\E[u_1(X_1)^2])+4\frac{n-2}{n-1}(\E[u_1(X_1)^2]-u_2^2).$ 
Moreover, 
\begin{talign}&n\|A_1
-\frac{4}{n(n-1)}[\mathbb{E}[u(X_1,X_2)^2]-\E[u_1(X_1)^2]]\|_2
\le\frac{12}{n-1}\|u(X_1,X_2)\|_2^2
\end{talign}
by Jensen's inequality. 
Using the triangle inequality, we also obtain that
\begin{talign}
n&\|A_2-\frac{4(n-2)}{n(n-1)}[\E[u_1(X_1)^2]-u_2^2]\|_2 
    \le
A_{21}+A_{22}+A_{23}\\
    &\defeq\frac{2}{(n-1)^2}\|\sum_{j_1\ne j_2>1}\E[u(X'_1,X_{j_1})u(X'_1,X_{j_2})\mid \X]-\E[u_1(X_1)^2]\|_2
    \\&+\frac{4}{n(n-1)^2}\|\sum_{j_1\ne j_2\ne i}u_1(X_{j_1})u(X_{i},X_{j_2})-u_2^2\|_2
    \\&+\frac{2}{n(n-1)^2}\|\sum_{j_1\ne j_2\ne i}u(X_i,X_{j_1})u(X_{i},X_{j_2})-\E(u_1(X_i)^2)\|_2.
\end{talign}

Since, for all $f\in L^2$, $\Cov(f(X_i, X_{j_1}, X_{j_2}), f(X_{i_2}, X_{j_3}, X_{j_4}))=0$ for distinct $i,i_2,j_1,j_2,j_3,j_4$,  
\begin{talign}
A_{21}
    &= \frac{2}{(n-1)^2}\sqrt{\sum_{\substack{j_1\ne j_2>1\\j_3\ne j_4>1}}\Cov(\E[u(X'_1,X_{j_1})u(X'_1,X_{j_2})\mid \X],\E[u(X'_1,X_{j_3})u(X'_1,X_{j_4})\mid \X])}\\&\le \frac{4\sqrt{n(n-1)(n-2)}}{(n-1)^2}\|u(X_1,X_2)\|_4^2
    \le  \frac{4\sqrt{n}}{n-1}\|u(X_1,X_2)\|_4^2, \\
A_{22}
    &=
\frac{4}{n(n-1)^2}\sqrt{\Var(\sum_{j_1\ne j_2\ne i}u_1(X_{j_1})u(X_{i},X_{j_2})-u_2^2)}
    \le
\frac{12}{\sqrt{n-1}}\|u(X_1,X_2)\|_4^2,
    \qtext{and}\\
A_{23}
    &=
\frac{2}{n(n-1)^2}\sqrt{\Var(\sum_{j_1\ne j_2\ne i}u(X_i,X_{j_1})u(X_{i},X_{j_2})-\E[u_1(X_i)^2])}
    \leq
\frac{6}{\sqrt{n-1}}\|u(X_1,X_2)\|_4^2
\end{talign}
by Jensen's inequality.
Since $n\geq 2$, we have 
$n\|A-\E(A)\|_2\le\frac{4\|u(X_1,X_2)\|_4^2}{\sqrt{n-1}}[\frac{3}{\sqrt{n-1}}+\frac{9}{2}+\sqrt{2}].$
Finally, our assumption $4\Cov(u(X_1,X_2),u(X_1,X_3))=1$ yields the second claim of \cref{second}: 
\begin{talign}
\|\frac{n}{2}\E[&F(\X^t,\X)(g(\X^t)-g(\X))\mid U_n]-1\|_2
    \le 
\|\frac{n}{2}(B+A-\E[B+A])+\frac{2(\E[u(X_1,X_2)^2]-u_2^2)}{n-1}\|_2 \\
    &\le
\frac{\|u(X_1,X_2)\|_4^2}{\sqrt{n-1}}[\frac{6}{\sqrt{n-1}}+16+2\sqrt{2}]
    +
\frac{2}{n-1}\|u(X_1,X_2)\|_4^2.
\end{talign} 

Now, fix any $t\ge-\frac{1}{2}\log(1-\frac{R^2}{\kappa n})$ and odd $k\ge 3$.  
By the triangle inequality,
\begin{talign}
&\|\E[F(\X^t,\X)(g(\X^t)-g(\X))^{k-1}\mid U_n]\|_2   \\
    &\le  
\frac{1}{2}\|\E[(g(\X^t)-g(\X))^{k}\mid U_n]\|_2  + \|\E[(g_1(\X^t)-
 g_1(\X))(g(\X^t)-g(\X))^{k-1}\mid U_n]\|_2.
\end{talign} 
Using Jensen's inequality, we find that 
\begin{talign}&
    \|  \E[(g(\X^t)-g(\X))^{k}\mid U_n]\|_2
    \le\frac{2^k}{\sqrt{n}^kn(n-1)^k} \|  \sum_{i\le n}(\sum_{j\ne i}u(X'_i,X_j)-u(X_i,X_j))^{k}\|_2.
\end{talign}
Moreover, for any fixed indices $i_1\neq i_2$ we have 
\begin{talign}&
|\E[ (\sum_{j\ne i_1}u(X'_{i_1},X_j)-u(X_{i_1},X_j))^{k}(\sum_{j\ne i_2}u(X'_{i_2},X_j)-u(X_{i_2},X_j))^{k}]|
\\&\overset{(a)}{\le }\| (\sum_{j\ne i_1}u(X'_{i_1},X_j)-u(X_{i_1},X_j))^{k}-(\sum_{j\notin\{i_1,i_2\}}u(X'_{i_1},X_j)-u(X_{i_1},X_j))^{k}\|_2^2
\\&\overset{(d)}{\le }k 2^{2(k-1)}(n-1)^{2(k-1)}R^{2(k-1)}\| u(X_{i_1},X_{i_2})-u(X'_{i_1},X_{i_2}) \|_2^2
\le
k 2^{2k}(n-1)^{2(k-1)}R^{2k}. %
\end{talign} where (a) follows from $(X_i')_{i=1}^n$ being an independent copy of $\X$ and the implied identities
\begin{talign}
&\E[ (\sum_{j\notin\{i_1,i_2\}}u(X'_{i_1},X_j)-u(X_{i_1},X_j))^{k}(\sum_{j\ne i_2}u(X'_{i_2},X_j)-u(X_{i_2},X_j))^{k}]=0
    \qtext{and} \\
&\E[ (\sum_{j\notin\{i_1,i_2\}}u(X'_{i_1},X_j)-u(X_{i_1},X_j))^{k}(\sum_{j\notin\{i_1,i_2\}}u(X'_{i_2},X_j)-u(X_{i_2},X_j))^{k}]
    =0,
\end{talign}
while (b) uses the relation $|a^k-b^k|\le k|a-b|\max(|a|^{k-1},|b|^{k-1})$ for all $a,b\in\mathbb{R}$.
Hence,
\begin{talign}&
 \| \sum_{i\le n}(\sum_{j\ne i}u(X'_i,X_j)-u(X_i,X_j))^{k}\|^2_2
 \le\sum_{i\le n}\E[(\sum_{j\ne i}u(X'_i,X_j)-u(X_i,X_j))^{2k}] \\
&+\sum_{i_1\ne i_2}\E[ (\sum_{j\ne i_1}u(X'_{i_1},X_j)-u(X_{i_1},X_j))^{k}(\sum_{j\ne i_2}u(X'_{i_2},X_j)-u(X_{i_2},X_j))^{k}]
\\&\le 2^{2k} n(n-1)^{2k}\|u(X_1,X_2)^k\|^2_2+k2^{2k}n(n-1)^{2k-1}R^{2k}
\le 2^{2k} n(n-1)^{2k-1}[n-1+k]R^{2k}.
\end{talign}  
Putting the pieces together, we obtain the estimate
\begin{talign}&\label{eq:gboundodd}
  n  \|  \E[(g(\X^t)-g(\X))^{k}\mid U_n]\|_2
    \le\frac{1}{\sqrt{n}^{k-1}} 2^{k} [1+\frac{\sqrt{k}}{\sqrt{n-1}}]R^{k}
    \le\frac{c(t)^{k-3}}{n} 2^{k} [1+\frac{\sqrt{k}}{\sqrt{n-1}}]R^{3}.
\end{talign}
Parallel reasoning yields the bounds
\begin{talign}&\label{eq:g1boundodd}
    n\|\E[(g_1(\X^t)-g_1(\X))(g(\X^t)-g(\X))^{k-1}\mid U_n]\|_2\\&\le\frac{2^{k-1}}{\sqrt{n}^{k}(n-1)^{k-1}}\|\sum_{i\le n}(u_1(X'_i)-u_1(X_i))(\sum_{j\ne i}u(X'_i,X_j)-u(X'_i,X_j))^{k-1}\|_2 \\
&\le\frac{1}{\sqrt{n}^{k-1}} 2^{k} [1+\frac{\sqrt{k}}{\sqrt{n-1}}]R^{k}
    \le\frac{c(t)^{k-3}}{n} 2^{k} [1+\frac{\sqrt{k}}{\sqrt{n-1}}]R^{3}.
\end{talign}
and the advertised odd $k$ result in \cref{higher}.

The even $k$ result \cref{higher} follows from Jensen's inequality and the triangle inequality as
\begin{talign}
\label{eq:gboundeven}
n&\|  \E[(g(\X^t)-g(\X))^{k}\mid U_n]\|_2
    \le \|  \frac{1}{\sqrt{n}^k(n-1)^k}\sum_{i\le n}(\sum_{j\ne i}u(X'_i,X_j)-u(X_i,X_j))^{k}\|_2
    \\&\le \frac{1}{\sqrt{n}^{k-2}(n-1)^k}\|  (\sum_{j\ne i}u(X'_1,X_j)-u(X_1,X_j))^{k}\|_2    
    \le \frac{2^k}{\sqrt{n}^{k-2}}R^{k}
    \le  \frac{2^kc(t)^{k-4}}{n}R^{4}
    \qtext{and} \\
\label{eq:g1boundeven}
n&\|  \E[(g_1(\X^t)-g_1(\X))(g(\X^t)-g(\X))^{k-1}\mid U_n]\|_2
    \\&\le \frac{1}{\sqrt{n}^{k-2}(n-1)^{k-1}}\|[u_1(X_1')-u_1(X_1)]  (\sum_{j\ne i}u(X'_1,X_j)-u(X_1,X_j))^{k-1}\|_2    %
    \le  \frac{2^kc(t)^{k-4}}{n}R^{4}.
\end{talign}
Finally, the bounds \cref{second,eq:gboundodd,,eq:g1boundodd,,eq:gboundeven,,eq:g1boundeven}
imply that \cref{eq:exchangeable-pair-assumption} holds as 
\begin{talign}
\sum_{k=3}^{\infty}\frac{e^{-kt}k \|\E[(g(\X^t)-g(\X))^{k-1}F(\X^t,\mathbf{X})\mid U_n]\|_2}{\sqrt{k-1}!\sqrt{1-e^{-2t}}^k} 
	\leq
\sum_{k=3}^{\infty}
\frac{3}{2} \frac{k}{\sqrt{k-1}!} (1+\frac{\sqrt{k}}{\sqrt{n-1}})
(\frac{4R^2}{(e^{2t}-1)n})^{k/2}
	< \infty.
\end{talign}

\subsection{\pcref{porcel}}\label{proof_porcel}Using the definition of $\tilde\tau_{t,s}$ and the triangle inequality we obtain the upper estimate  
\begin{talign}
&\int_{-\frac{1}{2}\log(1-\frac{D}{\sqrt{n}})}^{\infty}\|\frac{e^{-2t}}{{1-e^{-2t}}}(Z^2-1)-\frac{e^{-t}}{\sqrt{1-e^{-2t}}}U_nZ-\tilde\tau_{t,s}\|_{2}
    \le (a_1)+(a_2)+(a_3)\\
    &\defeq
\int_{-\frac{1}{2}\log(1-\frac{D}{\sqrt{n}})}^{\infty}\frac{e^{-t}\|h_1(Z)\|_{2}}{\sqrt{1-e^{-2t}}}\|n\E[F(\X^t,\X)\mid U_n]+U_n\|_{2}dt\\&+\int_{-\frac{1}{2}\log(1-\frac{D}{\sqrt{n}})}^{\infty}\frac{e^{-2t}\|h_2(Z)\|_{2}}{{1-e^{-2t}}}\|\frac{n}{2}\E[F(\X^t,\X)(g(\X^t)-g(\X))\mid U_n]-1\|_{2}dt
\\&+\int_{-\frac{1}{2}\log(1-\frac{D}{\sqrt{n}})}^{\infty}\sum_{k\ge 3} \frac{e^{-tk}\|h_{k}(Z)\|_{2}}{k!(\sqrt{1-e^{-2t}})^{k}}n\|\E[F(\X^t,\X)(g(\X^t)-g(\X))^{k-1}\mid U_n]\|_{2}dt.
\end{talign}
We will bound each of these terms in turn. First, by \cref{second} of \cref{pair_moment_bounds}, we have  
$(a_1)=0.$
Next, by \cref{second} and the identity $\twonorm{h_2(Z)}^2 = \E[(Z^2-1)^2]=2$, we find that
\begin{talign}\label{snow2}
(a_2) 
    &\le 
\twonorm{h_2(Z)}B_{n}\int_{-\frac{1}{2}\log(1-\frac{D}{\sqrt{n}})}^{\infty}\frac{e^{-2t}}{{1-e^{-2t}}}dt
    = 
\frac{\sqrt{2}B_{n}}{2}\int^{{1-\frac{D}{\sqrt{n}}}}_{0}\frac{1}{{1-t}}dt
    =
\frac{\sqrt{2}{B_{n}}}{4}\log(n/D^2).
\end{talign}

Now fix any odd $k\ge 3$. 
Since $\kappa\geq \frac{R^2}{D\sqrt{n}}$, if $t\ge -\frac{1}{2}\log(1-\frac{D}{\sqrt{n}})$ then $t\ge -\frac{1}{2}\log(1-\frac{R^2}{n\kappa})$.  Therefore by \cref{higher} of \cref{pair_moment_bounds}, we have
\begin{talign}
&\int_{-\frac{1}{2}\log(1-\frac{D}{\sqrt{n}})}^{\infty}\frac{e^{-kt}}{\sqrt{{1-e^{-2t}}}^{k}}n\|\E[F(\X^t,\X)(g(\X^t)-g(\X))^{k-1}|\mid U_n]\|_{2}dt
   \\& \le 
\frac{3R^3}{2n}[1+\frac{\sqrt{k}}{\sqrt{n-1}}]\int_{-\frac{1}{2}\log(1-\frac{D}{\sqrt{n}})}^{\infty}\frac{2^ke^{-tk}c(t)^{k-3}}{\sqrt{1-e^{-2t}}^{k}}dt
    =
\frac{3R^3\sqrt{\kappa}^{k-3}}{2n}[1+\frac{\sqrt{k}}{\sqrt{n-1}}]\int_{-\frac{1}{2}\log(1-\frac{D}{\sqrt{n}})}^{\infty}\frac{e^{-tk}}{\sqrt{1-e^{-2t}}^{3}}dt
  \\&  =
\frac{3R^3\sqrt{\kappa}^{k-3}}{2n}[1+\frac{\sqrt{k}}{\sqrt{n-1}}]\int^{\sqrt{1-\frac{ D}{\sqrt{n}}}}_0\frac{x^{k-1}}{\sqrt{1-x^2}^{3}}dx
    \le
\frac{3R^3\sqrt{\kappa}^{k-3}}{2n}[1+\frac{\sqrt{k}}{\sqrt{n-1}}]\frac{n^{1/4}}{\sqrt{D}}.
\end{talign}
Two applications of Stirling's approximation \citep{robbins1955remark} imply, for all $m\in\mathbb{N}\setminus\{0\}$,
\begin{talign}\label{taylor_swift}
\sqrt{(2m+1)!} &\geq \sqrt{2m+1}\sqrt{\sqrt{2\pi(2m)} \cdot (2m/e)^{2m} \cdot \exp(\frac{1} {12(2m)+1})} 
        \geq \sqrt{2m+1} 2^m m! / (m^{1/4}\alpha), \\
\sum_{\substack{k\ge 3\\k~\rm{odd}}}\frac{2^k\sqrt{\kappa}^k\sqrt{k}}{\sqrt{k!}}
    &\le  
2\alpha \sqrt{\kappa}\sum_{{m\ge 1}}\frac{2^m\kappa^mm^{1/4}}{m!}
    \le 
2\alpha \sqrt{\kappa}\sum_{{m\ge 1}}\frac{2^m\kappa^m}{(m-1)!}
    \le  
4\alpha\sqrt{\kappa}^{3/2}e^{2\kappa},
    \qtext{and} \\
\sum_{\substack{k\ge 3\\k~\rm{odd}}}\frac{2^k\sqrt{\kappa}^k}{\sqrt{k!}}
    &\le  
2\alpha \sqrt{\kappa}\sum_{{m\ge 1}}\frac{2^m\kappa^mm^{1/4}}{m!\sqrt{2m+1}}
    \le  
2\alpha(e^{2\kappa}-1)\sqrt{\kappa/3}.
\end{talign}
Hence, using Tonelli's theorem \citep[Thm.~13.8]{Schilling_2005} to exchange the sum and integral and the Hermite inequality $\|h_k(Z)\|_{2}\le \sqrt{k!}$ \citep[Lem.~3]{bonis2020stein},  we find that 
    \begin{talign}&
     \int_{-\frac{1}{2}\log(1-\frac{D}{\sqrt{n}})}^{\infty}
      \sum_{\substack{k\ge 3:\\k \text{ odd}}} \frac{\|h_k(Z)\|_{2}}{{k!}}\frac{e^{-tk}}{\sqrt{{1-e^{-2t}}}^{k}}n\|\E[F(\X^t,\X)(g(\X^t)-g(\X))^{k-1}\mid U_n]\|_{2}dt \\
    &\le 
\frac{6\alpha R^3 e^{2\kappa}}{(n\kappa)^{3/4}\sqrt{D}\sqrt{n-1}}
    +
\frac{\sqrt{3}\alpha R^3(e^{2\kappa}-1)}{n^{3/4}\sqrt{D}\kappa}.
\end{talign} 

Next fix any even $k\ge 4$. By \cref{higher}, %
\begin{talign}&
   \int^{\infty}_{-\frac{1}{2}\log(1-\frac{D}{\sqrt{n}})}\frac{e^{-tk}}{\sqrt{1-e^{-2t}}^{k}}n\|\E[F(\X^t,\X)(g(\X^t)-g(\X))^{k-1}\mid U_n]\|_{2}dt
 \\&\le 
\frac{2^k3\sqrt{\kappa}^{k-4}R^4}{2n} \int^{\sqrt{1-\frac{D}{\sqrt{n}}}}_{0}\frac{x^{k-1}}{\sqrt{1-x^2}^4}dx
\le\frac{2^k3\sqrt{\kappa}^{k-4}R^4}{2n} \int^{\sqrt{1-\frac{D}{\sqrt{n}}}}_{0}\frac{1}{({1-x^2})^2}dx
\\&\le\frac{2^k3\sqrt{\kappa}^{k-4}R^4}{8n}[\frac{2\sqrt{n}\sqrt{1-\frac{D}{\sqrt{n}}}}{D}+\log(\frac{2}{1-\sqrt{1-D/\sqrt{n}}})] 
\le \frac{2^k3\sqrt{\kappa}^{k-4}R^4}{8n}[\frac{2\sqrt{n}\sqrt{1-\frac{D}{\sqrt{n}}}}{D}+\log(\frac{4\sqrt{n}}{D})] .
\end{talign}
A double invocation of Stirling's approximation now gives, for all $m\in \mathbb{N}\setminus\{0\}$,
\begin{talign}\label{taylor_swift2}
&\sqrt{(2m)!} \geq \sqrt{\sqrt{2\pi(2m)} \cdot (2m/e)^{2m} \cdot \exp(\frac{1} {12(2m)+1})} 
        \geq 2^m m! / (m^{1/4}\alpha)
    \qtext{and hence} \\
&\sum_{\substack{k\ge 4:\\k~\textup{even}}}\frac{2^k\sqrt{\kappa}^k}{\sqrt{k!}}
   \le \sum_{\substack{m\ge 2}}\frac{2^{2m}{\kappa}^m}{\sqrt{(2m)!}}
   \le \sum_{\substack{m\ge 2}}\frac{2^{m}{\kappa}^mm^{1/4}\alpha}{m!}
   \le 2\kappa  \alpha  (e^{2\kappa}-1).
\end{talign} 
Hence, we find that, by Tonelli's theorem, 
\begin{talign}\label{dia}
&\int_{-\frac{1}{2}\log(1-\frac{D}{\sqrt{n}})}
\sum_{\substack{k\ge 4:\\k \text{ even}}}\frac{\|h_{k}(Z)\|_{2}}{k!} 
\frac{e^{-tk}}{\sqrt{1-e^{-2t}}^{k}}n\|\E[F(\X^t,\X)(g(\X^t)-g(\X))^{k-1}\mid U_n]\|_{2}dx 
   \\
    &\le 
 \alpha  (e^{2\kappa}-1)\frac{3R^4}{4\kappa n}[\frac{2\sqrt{n}\sqrt{1-\frac{D}{\sqrt{n}}}}{D}+\log(\frac{4\sqrt{n}}{D})]. 
\end{talign} 
Our results and the relation $\frac{e^{2\kappa}-1}{\kappa} \leq 2e^{2\kappa}$ imply that $(a_1)+(a_2)+(a_3) \leq \beta$ as desired.

\section{\pcref{example}} \label{proof_example}
We first identify the score functions of $U_n.$
 Fix any $\Psi\in C_c^\infty$ and, for each $i\leq n$, define $\X_{-i} \defeq (X_1, \dots, X_{i-1},X_{i+1}, \dots, X_n)$ and $S_{-i}\defeq \sum_{j\neq i} X_j$. 
Then, for all distinct $i,j\in[n]$,
\begin{talign}&
\E[\frac{\score[X_1](X_i)}{S_{-i}}\Psi(U_n)]
    =   
\E[\E[\frac{\score[X_1](X_i)}{\sum_{j\ne i}X_j}\Psi(U_n)\mid\X_{-i}]]
    =
\frac{-\E[\Psi'(U_n)]}{\sqrt{n}(n-1)}
    \qtext{and} \\
&\E[(\frac{\score[X_1](X_i)\score[X_1](X_j)}{S_{-i}S_{-j}}-\frac{\score[X_1](X_j)}{S_{-j}^2S_{-i}})\Psi(U_n)]\\& = 
-\frac{1}{\sqrt{n}(n-1)}\E[\frac{\score[X_1](X_j)}{S_{-j}}\Psi'(U_n)]+\E[\frac{\score[X_1](X_j)}{S_{-i}S_{-j}^2}\Psi(U_n)]-\E[\frac{\score[X_1](X_j)}{S_{-j}^2S_{-i}}\Psi(U_n)]
   =
\frac{1}{n(n-1)^2}\E[\Psi''(U_n)]
\end{talign}
since $\score[X_1]$ is a score function.
Thus, $U_n$ admits the first and second-order score functions
\begin{talign}
\score[U_n](U_n)
    &=
\frac{n-1}{\sqrt{n}}\E[\sum_{i=1}^{n}\frac{\score[X_1](X_i)}{S_{-i}}\mid U_n] 
    \sstext{and}
I_{U_n}(U_n)
    =
\E[\sum_{\substack{i\le n,\\j\ne i}}\frac{\score[X_1](X_j)(\score[X_1](X_i)-1/S_{-j})}{S_{-i}S_{-j}/(n-1)}\mid U_n].
\end{talign}

Next, we bound $\|\score[U_n](U_n)\|_2$ %
using %
Jensen's inequality, the fact $\E[\score[X_1](X_1)\mid\X_{-1}]=0$,  the Efron-Stein inequality~\citep[(2.1)]{steele1986efron},  %
Cauchy-Schwarz, 
and $X_1\ge a$ a.s.\ in turn:
\begin{talign}
\|\score[U_n]( U_n)\|^2_{2}
    &\le 
\|\frac{{n-1}}{\sqrt{n}}\sum_{i=1}^{n}\frac{\score[X_1](X_i)}{S_{-i}}\|^2_{2} 
    = 
\rm{Var}(\frac{{n-1}}{\sqrt{n}}\sum_{i=1}^{n}\frac{\score[X_1](X_i)}{S_{-i}})
    \\&\le
\frac{(n-1)^2}{2} 
\|\frac{\score[X_1](X_n)-\score[X_1](X_{n+1})}{S_{-n}}+\sum_{j< n}(\frac{\score[X_1](X_j)}{S_{-j}}-\frac{\score[X_1](X_j)}{S_{-j}+X_{n+1}-X_n})\|^2_{2}
    \\&\le
(n-1)^2 
\|\frac{\score[X_1](X_n)-\score[X_1](X_{n+1})}{S_{-n}}\|_2^2
    +
(n-1)^3\sum_{j< n}\|\frac{\score[X_1](X_j)(X_{n+1}-X_n)}{S_{-j}(S_{-j}+X_{n+1}-X_n)}\|^2_{2}
    \\&{\le }
\frac{2(n-1)^2\|\score[X_1](X_1)\|_2^2}{(n-1)^2a^2} 
+\frac{2(n-1)^4\|X_1\|_2^2\|\score[X_1](X_1)\|_2^2}{(n-1)^4a^4} 
    \le 
\frac{2\|\score[X_1](X_1)\|_2^2}{a^2} (1
+\frac{\|X_1\|_2^2}{a^2}). 
\end{talign}

Moreover, Jensen's inequality, Cauchy-Schwarz, $X_1\geq a$ a.s., and \cref{dino_veg} give
\begin{talign}
&\frac{1}{(n-1)^2}\|I_{U_n}( U_n)\|^2_{2}  
    \le 
2
(\|\sum_{i\leq n, j\ne i}\frac{\score[X_1](X_j)}{S_{-j}^2S_{-i}}\|^2_{2}  
    + 
\|\sum_{i\leq n}\frac{\score[X_1](X_i)^2}{S_{-i}^2}\|^2_2
    +
\|\sum_{i\leq n}\frac{\score[X_1](X_i)}{S_{-i}}\|^4_4) \\
    &\le
\frac{3n^2}{(n-1)^4a^6}
\|\score[X_1](X_1)\|^2_{2} 
    +
\frac{3n^2}{(n-1)^4a^4}
\|\score[X_1](X_1)\|^4_{4} 
    +
16
(\|\sum_{i\leq n}\frac{\score[X_1](X_i)}{S}\|^4_4
    +
\|\sum_{i\leq n}\frac{\score[X_1](X_i)X_i}{S_{-i}S}\|^4_4)
\\
    &\le
\frac{3n^2}{(n-1)^4a^6}
\|\score[X_1](X_1)\|^2_{2} 
    +
\frac{3n^2}{(n-1)^4a^4}
\|\score[X_1](X_1)\|^4_{4} 
    +
\frac{144}{n^2a^4}\|\score[X_1](X_1)\|^4_4
    +
\frac{16}{(n-1)^4a^8}
\|X_1\score[X_1](X_1)\|^4_4.
\end{talign}

Furthermore, for $\mu\defeq\E[X_1]$, the triangle inequality and \cref{dino_veg} imply
\begin{talign}
\sqrt{n}(n-1)\|U_n\|_4
    &\le 
\norm{(\sum_{i\le n}X_i-\mu)^2}_4
    +
2\mu n\norm{\sum_{i\le n}X_i-\mu}_4 
    + 
\norm{\sum_{i\le n}X_i^2-\mu^2}_4 \\
    &\le
n(7\norm{X_1-\mu}_8^2
    +
\norm{X_1^2-\mu^2}_4)
    +
\sqrt{12n}\,n\,\mu\norm{X_1-\mu}_4
\end{talign}
so that, by Cauchy-Schwarz and our prior arguments,
\begin{talign}
&\|U_n\score[U_n]( U_n)\|_{2}
    \le 
\|\score[U_n]( U_n)\|_{4}\|U_n\|_4
    \le 
\frac{1}{n}\norm{\sum_{i\le n} \frac{\score[X_1](X_i)}{S_{-i}}}_4 \sqrt{n}(n-1)\|U_n\|_4 \\
    &\le
(\frac{72}{a^4}\|\score[X_1](X_1)\|^4_4
    +
\frac{8n^2}{(n-1)^4a^8}
\|X_1\score[X_1](X_1)\|^4_4)^{1/4}
(\sqrt{12} \mu \norm{X_1-\mu}_4 
    + 
\frac{7\norm{X_1-\mu}_8^2 + \norm{X_1^2-\mu^2}_4}{\sqrt{n}}).
\end{talign} 
Since $\Cov(u(X_1,X_2),u(X_1,X_3))=\sigma^2/4$, the final claim now follows from \cref{ustattheorem}.
\section{\pcref{concentration}}\label{proof_concentration}
Write $\density$ for the Lebesgue density of $S$ and define $h\defeq\density/\varphi$. 
The first claim follows from the triangle inequality and Cauchy-Schwarz as 
\begin{talign}
| P(S> u)-\Phi^c(u)|&=|\E[(h(Z)-1)\mathbb{I}(Z> u)]|
   =|\E[(\sqrt{h(Z)}-1)(\sqrt{h(Z)}+1)\mathbb{I}(Z> u)]|
   \\&\le |\E[(\sqrt{h(Z)}-1)\mathbb{I}(Z> u)]| +   |\E[(\sqrt{h(Z)}-1)\sqrt{h(Z)}\mathbb{I}(Z> u)]|
   \\&\le \sqrt{\E[(\sqrt{h(Z)}-1)^2]}\big[\sqrt{\E[\mathbb{I}(Z> u)]}+\sqrt{\E[h(Z)\mathbb{I}(Z> u)]}\big]
    \\&= \sqrt{\int_{-\infty}^{\infty}(\sqrt{\density(x)}-\sqrt{\varphi(x)})^2dx}\big[\sqrt{\Phi^c(u)}+\sqrt{\P(S > u)}\big]
    \\&= \sqrt{2}H(S,Z)\big[\sqrt{\Phi^c(u)}+\sqrt{P(S> u)}\big].
\end{talign}
The second claim follows analogously as 
\begin{talign}
| P(|S|> u)-2\Phi^c(u)|&=|\E[(h(Z)-1)\mathbb{I}(|Z|> u)]|
            \le \sqrt{2}H(S,Z)\big[\sqrt{2\Phi^c(u)}+\sqrt{P(|S|> u)}\big]. 
\end{talign}

\section{Additional Lemmas}

\begin{lemma}[Marcinkiewicz-Zygmund inequality {\citep[Thm.~2.1]{rio2009moment}}]\label{dino_veg} Let $(\mc{F}_i)_{i\ge1}$ be a filtration and let  $(\tilde X_i)_{i\ge1}$ be a sequence of $L^p$ random variables , for some $p\ge 2,$ that are such that $\E[\tilde X_i|\mc{F}_i]=0$. Then we have 
 \begin{talign}\|\sum_{i\le n}\tilde X_i\|_p\le \sqrt{p-1}\sqrt{\sum_{i=1}^n\|\tilde X_i\|^2_p}.\end{talign}
\end{lemma}

\begin{lemma}[Score function properties]
\label{It} 
Let $S$ be a random variable with convex support $\Omega_S\subseteq\mathbb{R}$ and  Lebesgue density $f_S\in C^1$.
The following statements hold true.
\begin{enumerate}[leftmargin=*]
    \item Any first-order score function or second-order score function $I_S$ for $S$ is %
$S$-almost everywhere unique amongst elements of $L^1(S)$ satisfying \cref{eq:score-assumps}.
    \item If  $f_S$ is twice differentiable with $\score[S]\defeq{f_S'/}{f_S}\in\Lp[1](S), I_S\defeq{f_S''/}{f_S}\in\Lp[1](S)$, and 
    \begin{talign}\label{eq:vanishing}
    \lim_{x\to\sup\Omega_S}\max(|f_S(x)|,|f'_S(x)|) = \lim_{x\to\inf\Omega_S}\max(|f_S(x)|,|f'_S(x)|) = 0,  
    \end{talign}
    then $\score[S]$ and $I_S$ are first and second-order score functions for $S$. %
\end{enumerate}
\end{lemma}
\begin{proof}
Suppose $g,\tilde{g}\in\Lp[1](S)$ are both score functions for $S$, 
fix $\epsilon>0$, and choose $\delta>0$ so that $\E[|g(S)-\tilde{g}(S)|\mathbb{I}(|S|> \delta)] < \epsilon$. 
Since any measurable compactly supported indicator function can be arbitrarily approximated by $\Cinfc$ functions, select $\Psi_{\epsilon,\delta}\in  \Cinfc$ satisfying 
$\|\Psi_{\epsilon,\delta}(S)-\mathbb{I}(g(S)> \tilde{g}(S), |S|\le \delta)\|_\infty\le \epsilon$
so that
\begin{talign}
\E[\max(g(S)-\tilde{g}(S),0)]
    &=
\E[[g(S)-\tilde{g}(S)]~\mathbb{I}(g(S)>\tilde{g}(S))]\\
    &\le 
\epsilon+\E[[g(S)-\tilde{g}(S)]\Psi_{\epsilon,\delta}(S)]+\E[|g(S)-\tilde{g}(S)|\mathbb{I}(|S|> \delta)]
    \le
2\epsilon.
\end{talign}
Since $\epsilon>0$ was arbitrary we  have $g(S)\le \tilde{g}(S)$ almost surely and, by symmetry, $g(S)\ge \tilde{g}(S)$ almost surely. %
Identical reasoning yields the claim for second-order scores.
    
Now suppose that $f_S$ is twice differentiable with $\score[S],I_S\in\Lp[1](S)$ and \cref{eq:vanishing} holding. %
For all $\Psi\in \Cinfc$, integration by parts, the boundedness of $\Psi$, and the assumption \cref{eq:vanishing} imply 
\begin{talign}
\E[\frac{f'_S(S)}{f_S(S)}\Psi(S)]
     &=
\int_{\Omega_S}f'_S(x)\Psi(x)dx
     =
f_S(x)\Psi(x)|_{\inf\Omega_S}^{\sup\Omega_S}- \int_{\Omega_S}f_S(x)\Psi'(x)dx
     =
-\E[\Psi'(S)] 
    \\
\E[\frac{f''_S(S)}{f_S(S)}\Psi(S)]
    &=
\int_{\Omega_S}f''_S(x)\Psi(x)dx
    =
f'_S(x)\Psi(x)|_{\inf\Omega_S}^{\sup\Omega_S}- \int_{\Omega_S}f'_S(x)\Psi'(x)dx \\
    &=
-\E[\frac{f'_S(S)}{f_S(S)}\Psi'(S)]
    =
\E[\Psi''(S)].
\end{talign} 
Hence $f'_S/f_S$ and $f''_S/f_S$ are first and second-order score functions.
\end{proof}

\begin{lemma}[Second-order score bound]\label{It2}
Suppose $S\overset{d}{=}S_1+S_2$ for independent random variables $S_1, S_2$ with first-order score functions $\score[S_1]\in\Lp[2](S_1)$ and $\score[S_2]\in \Lp[2](S_2)$. %
Then $S$ admits a second-order score function $I_S\in L^2(S)$ satisfying  \cref{eq:second-order-score-bound}.
\end{lemma}
\begin{proof} 
Fix any $\Psi \in C_c^{\infty}$.
We invoke the tower property, the independence of $S_1$ and $S_2$, and the first-order score function definition \cref{eq:score-assumps} in turn to discover that 
\begin{talign}
\E[\E[\score[S_1](S_1)\score[S_2](S_2)|S_1+S_2)\Psi(S_1+S_2)]
&=\E[\score[S_1](S_1)\score[S_2](S_2)\Psi(S_1+S_2)]\\
&=-\E[\score[S_1](S_1)\Psi'(S_1+S_2)]
=\E[\Psi''(S_1+S_2)].
\end{talign}
Hence $I_S(S) =\E[\score[S_1](S_1)\score[S_2](S_2)|S]$ satisfies the second-order score definition \cref{eq:score-assumps}. 
Finally, by Jensen's inequality and the independence of $S_1$ and $S_2$, we have %
$
\|I_S(S)\|_2=\|\E[\score[S_1](S_1)\score[S_2](S_2)|S_1+S_2]\|_2\le \|\score[S_1](S_1)\score[S_2](S_2)\|_2= \|\score[S_1](S_1)\|_2\|\score[S_2](S_2)\|_2.$
\end{proof}

\begin{lemma}[Fundamental theorem precondition]\label{espoir}
For all 
$x\in \Omega_S$, 
under the conditions of \cref{general_bound},  $t\rightarrow \sqrt{f_{S_t}(x)}$ is continuous on $[0,\infty)$ and continuously differentiable on $(0,\infty)$.
\end{lemma}
\begin{proof}
Fix any $x\in \Omega_S$.
By \citet[Thm.~4.1]{pavliotis2014stochastic}, $t\mapsto f_{S_t}(x)$ is continuous on $[0,\infty)$ and continuously differentiable on $(0,\infty)$.
Thus, $t\mapsto \sqrt{f_{S_t}(x)}$ is  continuous on $[0,\infty)$. 

Next, fix any $t > 0$.
To conclude, it remains to show that $f_{S_t}(x)>0$ so that $t\mapsto \sqrt{f_{S_t}(x)}$ is also continuously differentiable on $(0,\infty)$. 
Since $f_S$ is continuous, $\exists \delta > 0$ for which 
$$\inf_{y\in [x-\delta,x+\delta ]} f_S(y)>{f_S(x)}{/2}.$$ 
Now consider the interval  
$M_x\defeq x(1-e^{-t})+\delta e^{-t}\sqrt{1-e^{-2t}}[-1,1]$.
Since $f_S(x) > 0$, %
\begin{talign}
  f_{S_t}(x)
  &  =
\int_{\reals} \frac{\varphi(y/\sqrt{1-e^{-2t}})}{\sqrt{1-e^{-2t}}} \frac{f_{S}((x-y)/e^{-t})}{e^{-t}} dy  
\ge \int_{M_x}\frac{\varphi(y/\sqrt{1-e^{-2t}})}{\sqrt{1-e^{-2t}}} \frac{f_{S}((x-y)/e^{-t})}{e^{-t}} dy  
\\&\overset{(a)}{\ge} \frac{\varphi(|x|+\delta )}{\sqrt{1-e^{-2t}}}\int_{M_x} \frac{f_{S}((x-y)/e^{-t})}{e^{-t}} dy
\overset{(b)}{\ge} \varphi(|x|+\delta )\delta f_S(x) > 0,  
\end{talign} 
as $|y/\sqrt{1-e^{-2t}}|\le |x|+\delta $ for all $y\in M_x$, yielding (a), and
$|(x-y)/e^{-t}-x|\le \delta $ for all $y\in M_x$ so that $f_S((x-y)/e^{-t})\ge f_S(x)/2$,  yielding (b).
\end{proof}

\begin{lemma}[Fubini precondition]\label{marreee}
Under the conditions of \cref{general_bound}, we have 
\begin{talign}&  \int_{0}^{T}\int_{\Omega_S}\sqrt{f_S(x)}|\partial_t\sqrt{f_{S_t}(x)}|dxdt<\infty
\qtext{for all} T > 0.
\label{eq:fubini-precondition}
\end{talign}
\end{lemma}
\begin{proof}
Fix any $T > 0$. 
The identities  \cref{oma,,eq:secondderivloghtS,,eq:derivloghtS}, show that
\begin{talign} |\partial_t\sqrt{f_{S_t}(x)}|=\frac{\sqrt{f_{S_t}(x)}}{2}|I_t^S(x)+x\score[S_t](x)+1|
\qtext{for all} t\in [0,T] 
\qtext{and} x\in\reals.
\end{talign} 
Allowing for the integral and expectation of nonnegative quantities to be infinite, we find
\begin{talign}&
\int_{0}^{T}\int_{\Omega_S}\sqrt{f_S(x)}|\partial_t\sqrt{f_{S_t}(x)}|\,dx\,dt
    =
\half\int_{0}^{T}\E[\frac{\sqrt{f_S(S_t)}}{\sqrt{f_{S_t}(S_t)}}|I_t^S(S_t)+S_t\score[S_t](S_t)+1|] dt.
\end{talign} 
Moreover, the almost sure identities   \cref{mdt,jst1} and Cauchy-Schwarz imply
\begin{talign}&  \half\E[{\sqrt{f_S(S_t)}}{}
|I_t^S(S_t)+S_t\score[S_t](S_t)+1|/\sqrt{f_{S_t}(S_t)}]
\\&\le \half\|{\sqrt{f_S(S_t)}/}{\sqrt{f_{S_t}(S_t)}}\|_2[1+e^{2T}\|I_S(S)\|_2+2e^T\|\score[S](S)\|_2+e^T\|\score[S](S)S\|_2]
\end{talign}
Since $I_S,\score\in\Lp[2](S)$ and $\twonorm{\score(S)S}<\infty$ the result \cref{eq:fubini-precondition} follows from the observation
\begin{talign}
\half\|{\sqrt{f_S(S_t)}/}{\sqrt{f_{S_t}(S_t)}}\|_2=\frac{1}{2}\sqrt{\int_{\Omega_S}f_S(x)dx}=\frac{1}{{2}}.
\end{talign}
\end{proof}

\bibliographystyle{abbrvnat}
\bibliography{refs}

\begin{thebibliography}{23}
\providecommand{\natexlab}[1]{#1}
\providecommand{\url}[1]{\texttt{#1}}
\expandafter\ifx\csname urlstyle\endcsname\relax
  \providecommand{\doi}[1]{doi: #1}\else
  \providecommand{\doi}{doi: \begingroup \urlstyle{rm}\Url}\fi

\bibitem[Artstein et~al.(2004)Artstein, Ball, Barthe, and
  Naor]{artstein2004rate}
S.~Artstein, K.~M. Ball, F.~Barthe, and A.~Naor.
\newblock On the rate of convergence in the entropic central limit theorem.
\newblock \emph{Probability theory and related fields}, 129\penalty0
  (3):\penalty0 381--390, 2004.

\bibitem[Austern and Mackey(2022)]{austern2022efficient}
M.~Austern and L.~Mackey.
\newblock Efficient concentration with {G}aussian approximation.
\newblock \emph{arXiv preprint arXiv:2208.09922}, 2022.

\bibitem[Barron(1986)]{barron1986entropy}
A.~R. Barron.
\newblock Entropy and the central limit theorem.
\newblock \emph{Ann. Probab.}, pages 336--342, 1986.

\bibitem[Bobkov and G{\"o}tze(2025{\natexlab{a}})]{bobkov2025central}
S.~G. Bobkov and F.~G{\"o}tze.
\newblock Central limit theorem for r{\'e}nyi divergence of infinite order.
\newblock \emph{The Annals of Probability}, 53\penalty0 (2):\penalty0 453--477,
  2025{\natexlab{a}}.

\bibitem[Bobkov and G{\"o}tze(2025{\natexlab{b}})]{bobkov2025renyi}
S.~G. Bobkov and F.~G{\"o}tze.
\newblock R{\'e}nyi divergences in central limit theorems: Old and new.
\newblock \emph{Probability Surveys}, 22:\penalty0 1--75, 2025{\natexlab{b}}.

\bibitem[Bobkov et~al.(2013)Bobkov, Chistyakov, and G{\"o}tze]{bobkov2013rate}
S.~G. Bobkov, G.~P. Chistyakov, and F.~G{\"o}tze.
\newblock Rate of convergence and edgeworth-type expansion in the entropic
  central limit theorem.
\newblock \emph{The Annals of Probability}, pages 2479--2512, 2013.

\bibitem[Bobkov et~al.(2014)Bobkov, Chistyakov, and G{\"o}tze]{bobkov2014berry}
S.~G. Bobkov, G.~P. Chistyakov, and F.~G{\"o}tze.
\newblock {B}erry--{E}sseen bounds in the entropic central limit theorem.
\newblock \emph{Probability Theory and Related Fields}, 159\penalty0
  (3-4):\penalty0 435--478, 2014.

\bibitem[Bonis(2020)]{bonis2020stein}
T.~Bonis.
\newblock {S}tein’s method for normal approximation in wasserstein distances
  with application to the multivariate central limit theorem.
\newblock \emph{Probab. Theory Relat. Field}, 178\penalty0 (3):\penalty0
  827--860, 2020.

\bibitem[Goldfeld and Kato(2020)]{goldfeld2020limit}
Z.~Goldfeld and K.~Kato.
\newblock Limit distributions for smooth total variation and $\chi$
  2-divergence in high dimensions.
\newblock In \emph{ISIT}, pages 2640--2645. IEEE, 2020.

\bibitem[Hellinger(1909)]{hellinger1909neue}
E.~Hellinger.
\newblock Neue begründung der theorie quadratischer formen von unendlichvielen
  veränderlichen.
\newblock \emph{Journal für die reine und angewandte Mathematik},
  1909\penalty0 (136):\penalty0 210--271, 1909.

\bibitem[Janson et~al.(2024)Janson, Pratelli, and Rigo]{janson2024quantitative}
S.~Janson, L.~Pratelli, and P.~Rigo.
\newblock Quantitative bounds in the central limit theorem for m-dependent
  random variables.
\newblock \emph{ALEA}, 21:\penalty0 245--265, 2024.

\bibitem[Knapp(2007)]{knapp2007basic}
A.~W. Knapp.
\newblock \emph{Basic real analysis}.
\newblock Springer Science \& Business Media, 2007.

\bibitem[Krantz and Parks(2002)]{krantz2002primer}
S.~Krantz and H.~Parks.
\newblock \emph{A Primer of Real Analytic Functions}.
\newblock Birkh{\"a}user Boston, 2002.

\bibitem[Ledoux et~al.(2015)Ledoux, Nourdin, and Peccati]{ledoux2015stein}
M.~Ledoux, I.~Nourdin, and G.~Peccati.
\newblock {S}tein’s method, logarithmic {S}obolev and transport inequalities.
\newblock \emph{Geometric and Functional Analysis}, 25\penalty0 (1):\penalty0
  256--306, 2015.

\bibitem[Pavliotis(2014)]{pavliotis2014stochastic}
G.~A. Pavliotis.
\newblock Stochastic processes and applications.
\newblock \emph{Texts in applied mathematics}, 60, 2014.

\bibitem[Pratelli and Rigo(2018)]{pratelli2018convergence}
L.~Pratelli and P.~Rigo.
\newblock Convergence in total variation to a mixture of {G}aussian laws.
\newblock \emph{Mathematics}, 6\penalty0 (6):\penalty0 99, 2018.

\bibitem[Rio(2009)]{rio2009moment}
E.~Rio.
\newblock Moment inequalities for sums of dependent random variables under
  projective conditions.
\newblock \emph{Journal of Theoretical Probability}, 22\penalty0 (1):\penalty0
  146--163, 2009.

\bibitem[Robbins(1955)]{robbins1955remark}
H.~Robbins.
\newblock A remark on {S}tirling's formula.
\newblock \emph{Am. Math. Mon.}, 62\penalty0 (1):\penalty0 26--29, 1955.

\bibitem[Ross(2011)]{ross2011fundamentals}
N.~Ross.
\newblock Fundamentals of {S}tein’s method.
\newblock \emph{Probability Surveys}, 8:\penalty0 210--293, 2011.

\bibitem[Schilling(2005)]{Schilling_2005}
R.~L. Schilling.
\newblock \emph{Measures, Integrals and Martingales}.
\newblock Cambridge University Press, 2005.

\bibitem[Steele(1986)]{steele1986efron}
J.~M. Steele.
\newblock {An Efron-Stein Inequality for Nonsymmetric Statistics}.
\newblock \emph{Ann. Stat.}, 14:\penalty0 753--758, 1986.

\bibitem[Stein(1986)]{stein1986approximate}
C.~Stein.
\newblock Approximate computation of expectations.
\newblock \emph{Lecture Notes Monogr. Ser.}, 7:\penalty0 i--164, 1986.

\bibitem[Yao and Liu(2025)]{yao2025symmetric}
L.-Q. Yao and S.-H. Liu.
\newblock Symmetric {KL}-divergence by {S}tein’s method.
\newblock \emph{Stoch Proc Appl.}, 185, 2025.

\end{thebibliography}

\begin{acks}
MA was supported by ONR N000142112664 and NSF DMS-2441652. 
\end{acks}

\end{document}

